\documentclass[11pt,a4paper]{amsart}

\usepackage[utf8]{inputenc}
\usepackage[T1]{fontenc}
\usepackage[italian,english]{babel}
\usepackage{setspace}
\usepackage{hyperref}
\hypersetup{colorlinks=true,linkcolor=blue,urlcolor=blue,citecolor=blue}

\usepackage{amssymb}
\usepackage{rotating}

\usepackage{enumitem}
\usepackage[all]{xy}

\usepackage{rotating}

\usepackage{tabularx}
\usepackage{booktabs}
\usepackage{caption}
\usepackage{geometry}

\renewcommand{\epsilon}{\varepsilon}
\renewcommand{\phi}{\varphi}

\numberwithin{equation}{section}
\numberwithin{figure}{section}
\numberwithin{table}{section}

\usepackage{pgf,tikz,pgfplots}
\usepackage{mathrsfs}
\usetikzlibrary{arrows}

\theoremstyle{theorem}
\newtheorem{theorem}{Theorem}[section]
\newtheorem{lemma}[theorem]{Lemma}
\newtheorem{corollary}[theorem]{Corollary}
\newtheorem{proposition}[theorem]{Proposition}

\newtheorem{notation}[theorem]{Notation}

\theoremstyle{definition}

\newtheorem{definition}[theorem]{Definition}

\newtheorem{remark}[theorem]{Remark}

\newtheorem{construction}[theorem]{Construction}

\def\C{{\mathbb C}}

\def\O{{\mathbb O}}
\def\P{{\mathbb P}}

\def\R{{\mathbb R}}
\def\Z{{\mathbb Z}}

\def\cD{{\mathcal D}}

\def\cO{{\mathcal{O}}}
\def\cP{{\mathcal{P}}}
\def\cQ{{\mathcal{Q}}}

\def\cY{{\mathcal Y}}

\def\fg{{\mathfrak g}}
\def\fh{{\mathfrak h}}
\def\fp{{\mathfrak p}}
\def\fb{{\mathfrak b}}

\def\SL{\operatorname{Sl}}

\def\SP{\operatorname{Sp}}

\def\conj{\operatorname{conj}}

\def\DA{{\rm A}}
\def\DB{{\rm B}}
\def\DC{{\rm C}}
\def\DD{{\rm D}}
\def\DE{{\rm E}}
\def\DF{{\rm F}}
\def\DG{{\rm G}}

\def\Aut{\operatorname{Aut}}

\def\Hom{\operatorname{Hom}}

\def\Pic{\operatorname{Pic}}

\def\id{\operatorname{id}}

\def\rank{\operatorname{rank}}

\def\deg{\operatorname{deg}}

\def\Comp{\operatorname{Comp}}
\def\Lie{\operatorname{Lie}}

\makeindex

\begin{document}

\title[Minimal BW for $\C^*$-actions on generalized Grassmannians]{Minimal bandwidth for $\C^*$-actions on generalized Grassmannians}

\author[Alberto Franceschini]{Alberto Franceschini}
\thanks{Supported by the PhD school of the Department of Mathematics of University of Trento}
\address{Dipartimento di Matematica, Universit\'a di Trento, via
Sommarive 14, I-38123 Povo di Trento (TN), Italy}
\email{alberto.franceschini@unitn.it}


\begin{abstract}
The bandwidth of a $\C^*$-action of a polarized pair $(X,L)$ is a natural measure of its complexity. In this paper, we study $\C^*$-actions on rational homogeneous spaces, determining which provide minimal bandwidth. We prove that the minimal bandwidth is linked to the smallest coefficient of the fundamental weight, in a base of simple roots, which describes the variety as a marked Dynkin diagram. As a direct application of the results we study the Chow groups of the Cayley plane $\DE_6(6)$.
\end{abstract}

\maketitle


\section{Introduction}\label{sec:intro}

GKM-theory was introduced in \cite{article:gkm} to study the topology of a manifold $M$, using the \emph{skeleton structure} of fixed points and 1-dimensional orbits provided by the action of a torus $S^1 \times \ldots \times S^1$.
On the algebraic side, an analogue theory for complex projective varieties has been developed in \cite{BuWeWi}, using algebraic tori $\C^* \times \ldots \times  \C^*$.

\subsection{Contents of the paper}
Rational homogeneous varieties (\emph{RH varieties} for short) appear to be the most natural examples to test the tools of the theory. 
Given a semisimple group $G$, we may consider the restriction to a maximal torus $H \subset G$ of the natural $G$-action on the RH variety $G/P$. The skeleton structure for the $H$-action on $G/P$ consists of a polytope $\Delta(G/P) \subset \R^n$ of isolated fixed points and information about the decomposition of the (co)tangent space at these points.
Then, chosen a $1$-dimensional torus $\C^* \subset H$, the action will be encoded on the projection of the polytope to a particular $1$-dimensional lattice (see Construction \ref{con:downgrading}). The extremal values of the projection of $\Delta(G/P)$ are denoted by $\mu_{\text{max}}$ and $\mu_{\text{min}}$. The \emph{bandwidth} of the action is defined to be $|\mu|=\mu_{\text{max}}-\mu_{\text{min}}$ (see \cite{RomanoW} for the original paper).

It seems natural that the projection which provides the minimal bandwidth must be linked to the symmetries of the polytope $\Delta(G/P)$, which is given by the Coxeter polytope associated to the marked Dynkin diagram (see \cite{montagard2009regular} for an introduction to the topic). The symmetries of $\Delta(G/P)$ are given by some hyperplanes $h_i$ orthogonal to a base of \emph{simple roots} $\alpha_i$ of the root system $\Phi(G,H)$. 

We can restate the problem in terms of convex geometry.
	In fact, computing the minimal bandwidth for $\C^*$-actions on $X$ is equivalent to find a projection of minimum length of $\Delta(G/P)$ on a line. Moreover, in the language of \cite[Chapter 13]{polyCoxeter}, $\C^*$-actions of minimal bandwidth correspond to slicing the polytope $\Delta(G/P)$ with the minimum number of slices.
	\begin{figure}[ht!]
	\centering

\tikzset{every picture/.style={line width=0.75pt}} 

\begin{tikzpicture}[x=0.65pt,y=0.65pt,yscale=-1,xscale=1]

\draw  [dash pattern={on 4.5pt off 4.5pt}][line width=0.75]  (20,60) -- (140,60) -- (140,180) -- (20,180) -- cycle ;
\draw  [dash pattern={on 4.5pt off 4.5pt}][line width=0.75]  (70,20) -- (190,20) -- (190,140) -- (70,140) -- cycle ;
\draw  [fill={rgb, 255:red, 0; green, 0; blue, 0 }  ,fill opacity=1 ][line width=0.75]  (15,180) .. controls (15,177.24) and (17.24,175) .. (20,175) .. controls (22.76,175) and (25,177.24) .. (25,180) .. controls (25,182.76) and (22.76,185) .. (20,185) .. controls (17.24,185) and (15,182.76) .. (15,180) -- cycle ;
\draw  [fill={rgb, 255:red, 0; green, 0; blue, 0 }  ,fill opacity=1 ][line width=0.75]  (185,20) .. controls (185,17.24) and (187.24,15) .. (190,15) .. controls (192.76,15) and (195,17.24) .. (195,20) .. controls (195,22.76) and (192.76,25) .. (190,25) .. controls (187.24,25) and (185,22.76) .. (185,20) -- cycle ;
\draw [line width=2.25]    (20,60) -- (140,180) ;
\draw [line width=2.25]    (70,140) -- (140,180) ;
\draw [line width=2.25]    (20,60) -- (70,140) ;
\draw [line width=2.25]    (70,20) -- (190,140) ;
\draw [line width=2.25]    (70,20) -- (140,60) ;
\draw [line width=2.25]    (140,60) -- (190,140) ;
\draw  [dash pattern={on 4.5pt off 4.5pt}] (230,60) -- (350,60) -- (350,180) -- (230,180) -- cycle ;
\draw  [dash pattern={on 4.5pt off 4.5pt}] (280,20) -- (400,20) -- (400,140) -- (280,140) -- cycle ;
\draw [line width=2.25]    (280,140) -- (350,180) ;
\draw [line width=2.25]    (280,20) -- (350,60) ;
\draw [line width=2.25]    (230,60) -- (230,180) ;
\draw [line width=2.25]    (280,20) -- (280,140) ;
\draw [line width=2.25]    (350,60) -- (350,180) ;
\draw [line width=2.25]    (400,20) -- (400,140) ;
\draw  [line width=2.25]  (440,50) -- (560,50) -- (560,170) -- (440,170) -- cycle ;
\draw  [line width=2.25]  (490,10) -- (610,10) -- (610,130) -- (490,130) -- cycle ;
\draw [line width=0.75]  [dash pattern={on 4.5pt off 4.5pt}]  (20,60) -- (70,20) ;
\draw [line width=0.75]  [dash pattern={on 4.5pt off 4.5pt}]  (140,60) -- (190,20) ;
\draw [line width=0.75]  [dash pattern={on 4.5pt off 4.5pt}]  (20,180) -- (70,140) ;
\draw [line width=0.75]  [dash pattern={on 4.5pt off 4.5pt}]  (140,180) -- (190,140) ;
\draw [line width=0.75]  [dash pattern={on 4.5pt off 4.5pt}]  (230,60) -- (280,20) ;
\draw [line width=0.75]  [dash pattern={on 4.5pt off 4.5pt}]  (350,60) -- (400,20) ;
\draw [line width=0.75]  [dash pattern={on 4.5pt off 4.5pt}]  (230,180) -- (280,140) ;
\draw [line width=0.75]  [dash pattern={on 4.5pt off 4.5pt}]  (350,180) -- (400,140) ;
\draw [line width=0.75]  [dash pattern={on 4.5pt off 4.5pt}]  (440,50) -- (490,10) ;
\draw [line width=0.75]  [dash pattern={on 4.5pt off 4.5pt}]  (560,50) -- (610,10) ;
\draw [line width=0.75]  [dash pattern={on 4.5pt off 4.5pt}]  (440,170) -- (490,130) ;
\draw [line width=0.75]  [dash pattern={on 4.5pt off 4.5pt}]  (560,170) -- (610,130) ;

\end{tikzpicture}

\caption{
Slicing a cube: vertex-first, edge-first and face-first. 
The three ways of slicing the cube reflect the downgradings of the action of a maximal torus $H \subset \SP_6(\C)$ on the variety of $2$-dimensional isotropic subspaces of $\C^5$.}
	\end{figure}
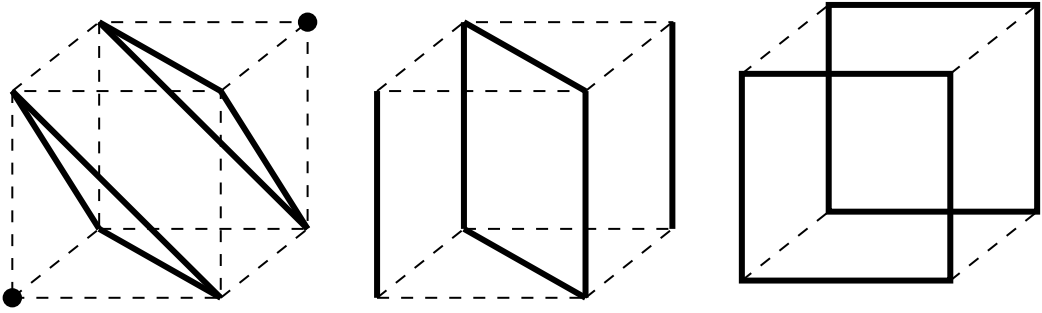

The main result of this paper is the following.

\begin{theorem}\label{thm:main}
	Let $\C^*$ act on the polarized pair $(X,L)$, where $X=\cD(j)$ is an RH variety with Picard number one given by a connected Dynkin diagram $\cD$ and one of its nodes $\alpha_j$, $L=\cO_X(1)$ is the (very) ample line bundle given by the fundamental weight $\omega_j$. Then the minimal bandwidth and the projections for which the minimal value is attained are given in Table \ref{tab:minBW}, using the notation developed in Section \ref{sec:Zgradings}.
	\begin{table}[ht!]
	\centering
	\caption{Minimal bandwidth for a $\C^*$-action on a generalized Grassmannian $\cD(j)$, where $\cD$ is the Dynkin diagram marked at the $j$-th node.}
	\begin{tabular}{c|c|c}
	\bf Generalized Grassmannian & \bf Minimal bandwidth & \bf Downgrading along $\mathbf{\alpha_i}$ \\
	\hline
	\hline
		$\DA_n(1) \simeq \DA_n(n)$ & $1$ & $\alpha_1,\ldots,\alpha_n$ \\
		$\DA_n(j)$ with $2 \le j \le n-1$ & $1,\ldots,1$ & $\alpha_1,\alpha_n$ \\
		\hline
		$\DB_n(1)$ & $2$ & $\alpha_1,\ldots,\alpha_n$ \\
		$\DB_n(j)$ with $j\ge 2$ & $2,\ldots,2,1$ & $\alpha_1$ \\
		\hline
		$\DC_n(1)$ & $1$ & $\alpha_n$ \\
		$\DC_n(2)$ & $2$ & $\alpha_1,\alpha_n$ \\
		$\DC_n(j)$ with $j\ge 3$ & $2,\ldots,2$ & $\alpha_1$ \\
		\hline
		$\DD_n(1)$ & $1$ & $\alpha_{n-1},\alpha_n$ \\
		$\DD_n(2)$ & $2$ & $\alpha_1,\alpha_{n-1},\alpha_n$ \\
		$\DD_n(j)$ with $3 \le j\le n-2$ & $2,\ldots,2$ & $\alpha_1$ \\
		$\DD_n(n-1) \simeq \DD_n(n)$ & $1$ & $\alpha_1,\alpha_4$ or $\alpha_3$ iff $n=4$\\
		\hline
		\hline
		$\DE_6(1) \simeq \DE_6(6)$ & $2$ & $\alpha_1,\alpha_2,\alpha_6$ \\
		$\DE_6(j)$ with $2 \le j \le 5$ & $2,3,4,3$ & $\alpha_1,\alpha_6$ \\
		\hline
		$\DE_7(j)$ with $j \le 5$ & $2,3,4,6,5$ & $\alpha_7$ \\
		$\DE_7(6)$ & $4$ & $\alpha_1,\alpha_7$ \\
		$\DE_7(7)$ & $2$ & $\alpha_1$ \\
		\hline
		$\DE_8(j)$ with $j \le 7$ & $4,6,8,12,10,8,6$ & $\alpha_8$ \\
		$\DE_8(8)$ & $4$ & $\alpha_1,\alpha_8$ \\
		\hline
		$\DF_4(1)$ & $4$ & $\alpha_1,\alpha_4$ \\
		$\DF_4(j)$ with $2 \le j \le 4$ & $6,4,2$ & $\alpha_1$ \\
		\hline
		$\DG_2(j)$ with $1 \le j \le 2$ & $2,4$ & $\alpha_2$ \\
		\hline
	\end{tabular}
	\label{tab:minBW}
\end{table}
\end{theorem}

\subsection{Why should this be interesting?}

Combining the statement of Theorem \ref{thm:main} with the Bia\l ynicki-Birula decomposition (see \cite{article:bb,article:bb2} for the original decomposition or \cite[Theorem 4.2]{Carrell2002}, \cite[Theorem 3.1]{BuWeWi}),
\begin{equation}\label{eq:decoHomBB}
			H_\bullet(X,\Z)= \bigoplus_{i=1}^s H_{\bullet-2\nu_+(i)} (Y_i,\Z)=\bigoplus_{i=1}^s H_{\bullet-2\nu_-(i)} (Y_i,\Z),
\end{equation}
(see Section \ref{sec:prelim} for the notation) one can study the homology (and, a posteriori, cohomology and Chow) groups of not so well-known RH varieties (for example, the ones arising from exceptional groups), using the simpler homology ring of its fixed-point components. We will show in Theorem \ref{thm:trasvComp} that the fixed-point components of the fundamental $\C^*$-actions (see Construction \ref{con:fundamental}) are varieties with the action of a certain semisimple subgroup $G' \subset G$.

When these $G'$-varieties are RH (that are almost all the cases), our result allows us to write $H_\bullet(G/P,\Z)$ as a sum of shifted homology rings on RH $G'$-varieties, with the minimum number of summands. This decomposition may be potentially useful for enumerative-geometric problems related to RH varieties of exceptional type, for which explicit homological descriptions are not so well-known.

\subsection{Outline of the paper}
Section \ref{sec:prelim} is devoted to an introduction to the theory of RH varieties and torus actions, in particular $\C^*$-actions.
In Section \ref{sec:Zgradings} we study the properties of fundamental $\C^*$-actions, that provide the projections on the lines orthogonal to the hyperplanes $h_i$. Also the proof of the Theorem \ref{thm:main} is contained in this section. Then, in Section \ref{sec:examples} we describe the $\C^*$-actions on classical cases, combining Representation Theory with techniques of Projective Geometry. Last, in Section \ref{sec:Cayley} we apply our results to describe the Chow groups of the Cayley plane $\DE_6(6)$, restating, in particular, some results presented on \cite{ManivelCayleyPlane}.

\subsubsection*{\bf Acknowledgements} 
The author wants to thank his advisor Luis Sol\'a Conde for his assistance and suggestions during the writing of this paper. He thanks also Eleonora Romano and Lorenzo Barban for helpful discussions on Friday afternoons.


\section{Preliminaries}\label{sec:prelim}

A \emph{projective variety} $X$ is an integral, proper, separated scheme of finite type over $\C$ that admits an ample line bundle $L \in \Pic(X)$.

\subsection{Rational homogeneous spaces}\label{ssec:RH}

In this section we provide a short introduction to RH varieties. We refer to \cite{Landsberg2003, manivel2020topics, campanaSurvey} for details and to \cite{fultonHarris} for the Representation Theory part.

Let $G$ be a semisimple algebraic group and let $H \subset B \subset P$ be, respectively, a maximal torus, a Borel subgroup and a parabolic subgroup of $G$. We denote by $\fh\subset\fb\subset\fp\subset\fg$ their corresponding Lie algebras. Moreover, we define $M(H):=\Hom(H,\C^*)$ to be the group of characters of $H$. As usual, given $\lambda \in M(H)$ and $t \in H$, we write $t^\lambda:=\lambda(t)$. 

The choice of $H$ and $B$ determine a \emph{root system} $\Phi=\Phi(G,H) \subset M(H)$, a base of \emph{simple roots} $\{\alpha_1,\ldots,\alpha_n\}$ and a decomposition $\Phi=\Phi^+ \cup \Phi^-$ onto \emph{positive and negative roots} such that $\fg=\fh \oplus \bigoplus_{\alpha \in \Phi} \fg_\alpha$ and $\fb=\fh \oplus \bigoplus_{\alpha \in \Phi^+} \fg_{\alpha}$, where $\fg_\alpha$ denotes the eigenspace associated to the root $\alpha$:
\[
\fg_\alpha=\left\{v \in \fg : t \cdot v= t^\alpha v \text{ for all }t \in H\right\}.
\]
We denote by $W=W(G,H)$ the \emph{Weyl group} of $G$.

Such a root system $\Phi$ can be associated uniquely to a Dynkin diagram $\cD$, whose nodes corresponds to a base $\{\alpha_1,\ldots,\alpha_n\}$ and the labels correspond to the angles between these simple roots (see \cite[Part III]{humLie}). Moreover, any parabolic subgroup $P$ is determined by its Lie algebra $\fp$ and one can prove (see \cite[p. 394]{fultonHarris}) that every parabolic Lie algebra is, up to conjugation, of the form
\[
\fp=\fb \oplus \bigoplus_{\alpha \in \Phi^+ \setminus \Phi^+(J)} \fg_{-\alpha} 
\]
where $\Phi^+(J)=\{ \alpha \in \Phi^+ : \exists\, \alpha_j \in J \text{ such that } \alpha-\alpha_j \in \Phi^+ \cup \{0\}\}$
and $J \subset \{\alpha_1,\ldots,\alpha_n\}$ is a subset of cardinality equals to the Picard number of $G/P$, hence we write $P=P_J$.

\begin{definition}
	An \emph{RH variety} is a projective variety on which $G$ acts transitively. It is a quotient $G/P_J$ and we can represent it as a marked Dynkin diagram, using the nodes that define $J$:
\[
	\cD(J):=G/P_J,
\]
where $\cD$ is the Dynkin diagram for the Lie algebra $\fg$.
\end{definition}


We recall that a \emph{weight} $\omega \in M(H) \otimes \R$ is an element 
such that $\alpha^\vee(\omega):=2(\omega,\alpha)/(\alpha,\alpha) \in \Z$ for all $\alpha \in \Phi$ and we denote the set of weights, called the \emph{weight lattice}, by $\cP_\Phi$. 
In general, it holds $\Z\Phi \subset M(H) \subset \cP_\Phi$. 

Moreover, $\omega$ is \emph{dominant} if and only if $\alpha^\vee(\omega) \ge 0$ for all $\alpha \in \Phi$. 
If we denote by $\omega_1,\ldots,\omega_n$ the \emph{fundamental weights}, i.e. such that $\alpha_i^\vee(\omega_j)=\delta_{ij}$, then a dominant weight can be written as $\omega=\sum_{j=1}^n m_j\omega_j$ with $m_j \ge 0$.
Since the Weyl group $W$ acts transitively on the Weyl chambers (see \cite[Theorem 10.3]{humLie}), up to change the basis using $W$, we can suppose that a weight $\omega$ is always dominant.

From now on we suppose that $G$ is the \emph{simply-connected group} of $\fg$, that is the semisimple group such that $M(H)=\cP_\Phi$. In particular $\omega_j \in M(H)$. Because of this choice, we obtain a $1:1$ correspondence (see \cite[Theorem 23.16]{fultonHarris}) between line bundles on $\cD(J)$ and irreducible $G$-representations. In fact, given a line bundle $L$, there exists a unique dominant weight $\omega=\sum_{\alpha_j \in J} m_j \omega_j \in M(H)$ such that
\[
H^0(\cD(J), L)=V(\omega),
\]
where $V(\omega)$ is the irreducible $G$-representation with \emph{highest weight} $\omega$, i.e. $\omega- \lambda$ is positive for every $\lambda \in M_\omega$. Moreover, $\cD(J)$ is the unique closed $G$-orbit contained in $\P(V(\omega))$ (see \cite[Claim 23.52]{fultonHarris}).
A \emph{fundamental $G$-representation} is $V=V(\omega_j)$ and the corresponding line bundle $L$ is the generator of the Picard group of $\cD(j)$, called a \emph{generalized Grassmannian}.

\subsection{Torus actions: grid data}\label{ssec:gridsData}

In this section we recall results that can be found in \cite[Section 2]{BuWeWi} and then in \cite{RomanoW,smallBandwidth,occhetta2020high}. 

Let $X$ be a smooth projective variety of dimension $d$ and let $L \in \Pic(X)$ be an ample line bundle. We consider a non-trivial action of an algebraic torus $H\simeq (\C^*)^n$ on $X$ and we write $X^H=Y_1 \sqcup \ldots \sqcup Y_s$ for the decomposition of the fixed-point locus into smooth connected components (see \cite{iversen}). Moreover, we define $\cY:=\{Y_1,\ldots,Y_s\}$.
Now we consider a linearization of the action $\mu_L^H: H \times L \to L$, that always exists (see \cite[Proposition 2.4]{article:kklv} and the rest of the paper for an introduction).

\begin{remark}
We define a \emph{weight map} (denoted again by) $\mu_L^H: \cY \to M(H)$ as follow: take a fixed point $p \in Y \subset X^H$, then the $H$-action on the fiber $L_p$ produces $\mu_L^H(p) \in M(H)$. Moreover, given another $p' \in Y$, we have that $\mu_L^H(p')=\mu_L^H(p)=\mu_L^H(Y)$ (see \cite{MR664117} for the symplectic point of view of the story).

\end{remark}

\begin{notation}
	An $H$-action on $(X,L)$ means an $H$-action on $X$ and a linearization $\mu_L$ of the action on $L$, which will be intended as the corresponding weight map $\mu_L:\cY \to M(H)$ defined as above.
\end{notation}

\begin{definition}
	The convex-hull of $\mu_L^H(Y_1),\ldots,\mu_L^H(Y_s)$ is called the \emph{polytope of fixed points} 
\[
\Delta(X)=\Delta(X,L,H,\mu_L) \subset M(H)_\R.
\]
\end{definition}

Again, let $p \in X^H$ and consider the induced $H$-action on the cotangent space $T_p X^\vee$. This action produces a set of characters $\nu_1(p),\ldots,\nu_d(p) \in M(H)$ and these characters are again constant on the fixed-point component $Y \supset p$, hence we write $\nu_1(Y),\ldots,\nu_d(Y)$.

\begin{definition}
	The set of non-zero characters among $\nu_1(Y),\ldots,\nu_d(Y)$ is called the \emph{compass} of the action, denoted by $\Comp(Y,X,H)$.
\end{definition}

Note that $\dim Y$ is equal to the number of zero characters among $\nu_i(Y)$. 


At this point we are interested in polytope of fixed points and compasses of generalized Grassmannians when $H$ is a maximal torus.

\begin{proposition}[\cite{occhetta2020high}, Lemmas 3.1, 3.3]\label{prop:gridsGP}
Let $(X,L)$ be a polarized pair, where $X=\cD(j)$ is a generalized Grassmannian and $L \in \Pic X$ is the generator. Consider the action of a maximal torus $H \subset P_j \subset G$:
\[
H \times X \ni (h, gP_j) \mapsto hgP_j \in X.
\]
Then the fixed points and their compasses can be obtained as follows
	\begin{description}[noitemsep]
		\item[Polytope of fixed points] $X^H=\{wP_j: w \in W\}$ and the character on the fiber over $wP_i$ is $-w(\omega_j)$;
		\item[Compass] $\Comp(wP_j,X,H)=\{w(\alpha): \alpha=\sum_i m_i(\alpha)\alpha_i \in \Phi^+ \text{ and } m_j(\alpha)>0\}$.
	\end{description}
\end{proposition}

We stress out that a linearization is always defined up to the choice of a character. Because of this fact, we will use the following linearization:
\begin{equation}\label{eq:tuttoIntero}
	\mu_L^H(wP_j):=\omega_j-w(\omega_j).
\end{equation}
Thanks to \cite[Corollary 2.18]{BuWeWi}, we obtain $\mu_L^H(wP_j) \in \Z \Phi$.

\begin{construction}\label{con:downgrading}
Consider a subtorus $H' \subset H$, this inclusion induces a contravariant map:
\begin{equation}\label{eq:downgrading}
	\imath^*: M(H) \to M(H').
\end{equation}

We denote by $\cY'$ the set of fixed-point components for the action of $H'$ and by $\mu_L^{H'}: \cY' \to M(H')$ the correspondent character map.
This operation is called a \emph{downgrading} of the action.
\end{construction}


From now on, we will be interested in the case $H' \simeq \C^*$, so $M(H')\simeq\Z$.

\begin{definition}\label{def:sinkSource}
The \emph{source} and the \emph{sink} of the $\C^*$-action are the unique fixed-point components $Y_+, Y_- \in \cY$ such that for a generic $x \in X$ we have, respectively,
\[
\lim_{t \to 0} t\cdot x \in Y_+ \quad \text{and}\quad \lim_{t \to 0} t^{-1}\cdot x \in Y_-.
\]
Alternatively, (see \cite[Lemma 2.12]{BuWeWi}) take a vertex $\mu'_L(Y)$ of $\Delta(X,L,\C^*,\mu'_L)$:
	\begin{itemize}[noitemsep]
		\item $\Comp(Y,X,\C^*) \subset \Z_{>0}$, then $Y$ is the sink.
		\item $\Comp(Y,X,\C^*) \subset \Z_{<0}$, then $Y$ is the source.
	\end{itemize}
\end{definition} 

\begin{definition}
	Consider a $\C^*$-action on $(X,L)$. The \emph{bandwidth} of the $\C^*$-action is defined by
\[
|\mu_L^{\C^*}|=\mu_L^{\C^*}(Y_+)-\mu_L^{\C^*}(Y_-).
\]
\end{definition}

We consider the induced $\C^*$-action on the tangent space $T X_{|Y}$, which decomposes as
\[
T X_{|Y}= TY \oplus N(Y|X)= N^+(Y|X) \oplus TY \oplus N^-(Y|X).
\]
where $\C^*$ acts with positive, zero and negative weights and we write
\[
\nu_\pm(Y):= \rank N^\pm(Y|X).
\]

Last, we introduce a special class of $\C^*$-actions ona polarized pair $(X,L)$: the equalized ones. This condition appears in the hypotheses of \cite[Theorem 8.1]{smallBandwidth} and will be central in future papers. In Section \ref{sec:examples} we show a few examples of non-equalized $\C^*$-actions.

\begin{definition}
A $\C^*$-action on $X$ is \emph{equalized at a fixed-point component $Y$} if all the weights on $N(Y|X)$ are $\pm1$. An action is \emph{equalized} if it is equalized at every component.
\end{definition}

\begin{remark}[\cite{smallBandwidth}, Remark 2.13]\label{rem:equalized}
Let $C$ be the closure of a 1-dimensional orbit for the $\C^*$-action and let $f: \P^1 \to C$ be its normalization. We lift up the $\C^*$-action on $\P^1$ obtaining two points $y_+,y_-$ as source and sink and we denote by $\delta(C)$ the weight of the lifted action on $T_{y_+}\P^1$. The $\C^*$-action is equalized if and only if $\delta(C)=1$ for every closure of 1-dimensional orbit.
\end{remark}

\begin{lemma}[\cite{smallBandwidth}, Lemma 2.12]\label{lem:AMvsFM2}
Let $X$ be a projective smooth variety with $L \in \Pic(X)$. Denote by $Y_+$ and $Y_-$ the sink and the source of a $\C^*$-action on $(X,L)$ and by $C$ the closure of a general orbit. Then
\begin{equation}\tag{AM vs FM}
\mu_L^{\C^*}(Y_+)-\mu_L^{\C^*}(Y_-)=\delta(C) \deg f^*L.
\end{equation}
\end{lemma}


\section{Downgradings along simple roots and associated $\C^*$-actions}\label{sec:Zgradings}

In this section we describe fundamental $\C^*$-actions on generalized Grassmannians, which is the main topic of this paper.

\subsection{$\Z$-gradings associated with simple roots}\label{ssec:gradings}

Let $X=G/P_j$ be a generalized Grassmannian and let $H \subset P_j \subset G$ be a maximal torus contained in the maximal parabolic subgroup $P_j$ of the simply-connected group $G$. 

\begin{construction}\label{con:fundamental}
Consider a $\C^*$-action on $X$, this is equivalent to give a map $\beta: \C^* \to \Aut(X)=G_{ad}$, where $G_{ad}$ is the adjoint group for the Lie algebra $\fg$. Thanks to the Jordan-Chevalley decomposition, we can restrict ourselves to consider the restriction of the map to a maximal torus
\[
\beta: \C^* \to H_{ad} \quad \Longrightarrow \quad \beta \in M(H_{ad})^\vee.
\]
Since $M(H_{ad})=\Z\Phi$, we obtain that $\beta$ can be thought as a map
\[
\beta: \Z \Phi \to M(\C^*) \simeq \Z.
\]
We will say that $H_i \subset H$ acts on $X$ with a \emph{fundamental $\C^*$-action} is the corresponding downgrading is as follow:
\begin{align*}
	\sigma_i: \Z \Phi & \to \Z \\
	\alpha_j & \mapsto \delta_{ij}.
\end{align*} 
Note that this construction generalize \cite[p. 38]{tevelev2006projective}. In particular, given a general $\C^*$-action on $X$, this will correspond to a linear combination $\sum_i q_i\sigma_i$ of fundamental $\C^*$-actions.

As remarked in Section \ref{sec:prelim}, up to change a change of basis, we can always suppose $q_i \ge 0$ (see \cite[p. 67]{humLie}).


\end{construction}


\subsection{The transversal group to a $\C^*$-action}\label{ssec:transversal}

Let $G$ be a semisimple group as above and let $H_i \subset H$ giving a fundamental $\C^*$-action on $X$ given by the downgrading along the simple root $\alpha_i$.
Consider the Cartan decomposition of $\fg=\fh \oplus \bigoplus_{\alpha \in \Phi} \fg_\alpha$, we obtain a $\Z$-grading on $\fg$ by setting
\[
\fg_0=\fh \oplus \bigoplus_{\substack{\alpha \in \Phi \\ m_i(\alpha)=0}} \fg_\alpha \quad \text{and} \quad \fg_m:=\bigoplus_{\substack{\alpha \in \Phi \\ m_i(\alpha)=m}} \fg_\alpha \text{ if }m \ne 0.
\]

\begin{remark}\label{rem:PZgrading}
We have that $\bigoplus_{m\ge 0} \fg_m$ is the Lie algebra of $P_i \subset G$:
\[
	\fp_i= \fh \oplus \bigoplus_{\alpha \in \Phi^+} \fg_{\alpha} \oplus \bigoplus_{\alpha \in \Phi^+ \setminus \Phi^+(i)} \fg_{-\alpha}=\fh \oplus \bigoplus_{\alpha \in \Phi^+} \fg_\alpha \oplus \bigoplus_{\substack{\alpha \in \Phi^+ \\ m_i(\alpha)=0}}\fg_{-\alpha}=\bigoplus_{m \ge 0} \fg_m.
\]
where we recall $\Phi^+(i)=\{ \alpha=\sum_{h=1}^n m_h(\alpha)\alpha_h \in \Phi^+: m_i(\alpha)>0\}$.
\end{remark}

The first part of the next statement can be found in \cite{tevelev2006projective}. The second part is a consequence of the previous remark, using the usual Levi decomposition of a parabolic subgroup.

\begin{proposition}[\cite{tevelev2006projective}, p.35]\label{thm:Tevelev}
	There exists  $G_0 \subset G$ reductive such that $\fg_0=\Lie(G_0)$, the Lie algebra associated to $G_0$. Moreover, $G_0$ is the Levi part of $P_i$ and the positive roots of the associated root system are $\Phi^+ \setminus \Phi^+(i)$.
\end{proposition}

It is a well-known fact that a reductive Lie algebra can be written as $\mathfrak{g}_0=\mathfrak{g}_0^\text{ss} \oplus \mathfrak{a}$ where $\mathfrak{g}_0^\text{ss}$ is semisimple and $\mathfrak{a}$ is abelian. 
We know that $\fh$ is generated by the coroots $\alpha_1^\vee, \ldots,\alpha_n^\vee$. Define $\mathfrak{h}^\perp$ as the Cartan algebra generated by all the coroots $\alpha^\vee_1,\ldots,\alpha^\vee_n$ but $\alpha_i^\vee$.
Then we can write
\[
\fg^\perp:=\fh^\perp \oplus \bigoplus_{\substack{\alpha \in \Phi^+ \\ m_i(\alpha)=0}} (\fg_{-\alpha} \oplus \fg_\alpha) \quad \text{and} \quad \fg_0=\fg^\perp \oplus \mathfrak{a}
\]
where $\fg^\perp$ is semisimple by construction and $\mathfrak{a}$ is the Lie algebra generated, as a vector space, by the coroot $\alpha_i^\vee$ (is abelian because it is a 1-dimensional Lie algebra).

\begin{definition}
	The \emph{transversal group} $G^\perp \subset G$ is the simply-connected semisimple algebraic group such that $\Lie(G^\perp)=\fg^\perp$.
\end{definition}

Let $V$ be an irreducible $G$-representation, which gives a natural representation of the Lie algebra $\fg \times V \to V$. We have again a $\Z$-grading induced by the downgrading along $\alpha_i$:
\[
	V = \bigoplus_{\lambda \in M(H)} V_\lambda =\bigoplus_{m \in \Z} V_m \quad \text{with } V_m=\bigoplus_{m_i(\lambda)= m} V_\lambda.
\]

\begin{remark}\label{rem2}
	Note that by construction every $V_m$ is $\fg^\perp$-invariant, hence a $G^\perp$-module.
\end{remark}

Now let $V=V(k_j\omega_j)$, hence $X\subset \P(V)$.
Again, the natural $G$-action  on $X$ descend to an $H$-action and then to a $H_i$-action. In particular
\[
X^{H_i}=X \cap \P(V)^{H_i}= \bigsqcup_{m \in \Z} (X \cap \P(V_m)).
\]
Let us note that in general $X \cap \P(V_m)$ is not necessarily connected. Consider the action of $G^\perp$: $X$ is $G'$-invariant and also $\P(V_m)$ is $G^\perp$-invariant, hence $X \cap \P(V_m)$ is $G^\perp$-invariant, so it is union of $G^\perp$-orbits. 

\begin{theorem}\label{thm:trasvComp}
It holds that $X \cap \P(V_m)$ is a finite union of disjoint $G^\perp$-orbits.
\end{theorem}

\begin{proof}
	It is sufficient to prove that:
	\begin{equation}\label{11}
		T_x (X \cap \P(V_m))= T_x (G^\perp \cdot x) \quad \text{for every } x \in X \cap \P(V_m).
	\end{equation}
	Let $x \in \tilde{X} \subset X \cap \P(V_m)$ where $\tilde{X}$ is the irreducible component of $X \cap \P(V_m)$ containing $x$. Note that $G^\perp \cdot x \subset \tilde{X}$. If they have the same dimension, then $G^\perp \cdot x$ is dense in $\tilde{X}$.
	
	We know that $G^\perp \cdot x$ is smooth, but $\tilde{X}$ might not be in general, hence:
	\[
	\dim T_x (G^\perp\cdot x)=\dim (G^\perp \cdot x) \le \dim \tilde{X} \le \dim T_x \tilde{X} \le \dim T_x (X \cap \P(V_m)).
	\]
	If $T_x (X \cap \P(V_m)) = T_x (G^\perp \cdot x)$, then $\dim \tilde{X}=\dim G^\perp \cdot x$.
	
	Suppose that $G^\perp \cdot x \subset \tilde{X}$ is dense. If $\tilde{X} \setminus G^\perp \cdot x \ne \emptyset$, $\dim G^\perp \cdot x < \dim \tilde{X}$, and it will be $G^\perp$-invariant.
	Take $x' \in \tilde{X} \setminus G^\perp \cdot x$. With the same argument one shows that $G^\perp \cdot x' \subset \tilde{X}$ is dense and we get a contradiction. So $\tilde{X}=G^\perp \cdot x$ for all $x \in \tilde{X}$. It follows that $X$ is disjoint union of $G^\perp$-closed orbits.
	
	Let us prove \eqref{11}: $x \in \P(V)$ is of the form $x=[v]$, then
	\begin{align*}
		T_x (G^\perp \cdot x)&=\frac{T_v (G^\perp\cdot v)}{\langle v \rangle}=\frac{\fg^\perp\cdot v}{\langle v \rangle}\\
		T_x (X \cap \P(V_m))&\subset T_x X \cap T_x \P(V_m)=\frac{\fg\cdot v}{\langle v \rangle} \cap \frac{V_m}{\langle v \rangle}=\frac{\fg\cdot v \cap V_m}{\langle v \rangle}.
	\end{align*}
	It remains to prove that $\fg^\perp\cdot v = \fg \cdot v \cap V_m$ for all $v \in V_m$, but this follows from Remark \ref{rem2}, hence we have $T_x(X \cap \P(V_m)) \subset T_x(G^\perp \cdot x)$.
	For the other inclusion, since the irreducible component of $X \cap \P(V_m)$ containing $x$ also contains $G^\perp \cdot x$, then $T_x(G^\perp \cdot x) \subset T_x(X \cap \P(V_m))$.
\end{proof}

\subsection{Proof of Main Theorem}

Let $(X,L)$ be a polarized pair with an $H$-action, where $X=\cD(j)=G/P_j$ and $L$ is the generator of the Picard group.

This section contains proof of Theorem \ref{thm:main}, that describes the $\C^*$-actions of minimal bandwidth of a generalized Grassmannian. The proof will be obtained in two steps. In the first one we will show that the minimal bandwidth is achieved for one of the fundamental $\C^*$-actions described in Section \ref{ssec:gradings}. In the second step we will show that the bandwidth of the $\C^*$-action associated to $\sigma_i$ can be computed as the difference of the weights of the action on $L$ at two given points, corresponding to the identity and to the longest element $w_\circ$ of the Weyl group of $G$.



The next statement shows that the minimal bandwidth for a $\C^*$-action on $(X,L)$ is obtained by one of the fundamental $\C^*$-actions $H_i$.

\begin{proposition}
	Let $H \subset G$ be a maximal torus. 
	Let $H' \subset H$ be a 1-dimensional subtorus acting (non trivially) on $(X,L)$ with weight map $\mu_L^{H'}: \cY' \to \Z$. Then there exists $i_0$ such that
	\[
	|\mu_L^{H'}| \ge |\mu_L^{i_0}|,
	\]
	where $\mu_L^{i_0}: \cY_{i_0} \to \Z$ is the weight map induced by the fundamental $\C^*$-action of $H_{i_0}$ on $(X,L)$.
\end{proposition}

\begin{proof}
	Because we have chosen the linearization as in \eqref{eq:tuttoIntero}, the weight map $\mu_L^H$ has its image contained in $\Z\Phi$.
	On the other hand, the $H'$-action on $X$ correspond, thanks to Construction \ref{con:downgrading}, to a map $\sum_i p_i \sigma_i: \Z \Phi \to \Z$. Hence, we can see the weight map of the $H'$-action on $X$ as
	\[
		\mu_L^{H'}=\left(\sum_i p_i \sigma_i\right) \circ \mu_L^H: \cY' \to \Z\Phi \to \Z.
	\]
	In particular, given a fixed-point $wP_j$, the map send it to
	\[
	wP_j \mapsto \omega_j -w(\omega_j)=\sum_{i} m_i \alpha_i \mapsto \sum_i p_i m_i.
	\]
	Because the action is non trivial, there exists at least one index $i_0$ such that $p_{i_0} \ge 1$.
\end{proof}


We finish the proof of Theorem \ref{thm:main} by looking at the value of the linearization over the fibers of two particular fixed points. 

\begin{remark}\label{rem:Pnega}
	In the decomposition of $\fg/\fp_j$ there are only negative roots of $\Phi$. Thanks to Remark \ref{rem:PZgrading} we have that $\fp_j=\fb \oplus \bigoplus_{\alpha \in \Phi^+ \setminus \Phi^+(j)} \fg_{-\alpha}$, hence
	\begin{equation}\label{eq:decog/p}
		\fg/\fp_j=\bigoplus_{\alpha-\alpha_j \in \Phi^+ \cup \{0\}} \fg_{-\alpha}.
	\end{equation}
\end{remark}
	
\begin{lemma}\label{lem:ePsource}
	Consider the fundamental $\C^*$-action of $H_i$ on $(X,L)$, then $eP_j \in Y_-$. Moreover, $eP_j=Y_-$ if and only if $i=j$.
\end{lemma}

\begin{proof}
	Let us consider $T_{eP_j} X=\fg/\fp_j$. Thanks to Remark \ref{rem:Pnega}, we consider the induced $H_i$-action on $v \in \fg_{-\alpha}$ where $\alpha=\sum_{h=1}^n a_h \alpha_h \in \Phi^+$ and $a_j \ge 1$:
	\[
	s \cdot v= s^{\sigma_i(\alpha)}v=s^{a_i}v.
	\]
	So we obtain that $H_i$ acts on $T_{eP_j}X$ with all non-negative weights (all positive iff $i=j$ because $a_j \ge 1$). By definition of compass, it follows that all its elements are non-positive (all negative iff $i=j$).
\end{proof}

\begin{remark}
	With the choice of the linearization of \eqref{eq:tuttoIntero}, it follows that $\mu_L^i(eP_j)=0$.
\end{remark}

From the general theory (see \cite[Section 1.8]{humCoxeter}), there exists $w_\circ \in W$ such that $w_\circ \Phi^+=\Phi^-$. In particular, if $W$ is not of type $\DA_n$ with $n>1$, $\DD_{2n+1}$ and $\DE_6$, then $w_\circ=-\id$.
Otherwise
\begin{align}
	\tag{$\DA_n$} &w_\circ \alpha_i=-\alpha_{n+1-i}; \\
	\tag{$\DD_{2n+1}$} &w_\circ \alpha_{2n+1}=-\alpha_{2n}; \\
	\tag{$\DE_6$} &w_\circ\{\alpha_1,\alpha_2,\alpha_3,\alpha_4,\alpha_5,\alpha_6\}=\{-\alpha_6,-\alpha_2,-\alpha_5, -\alpha_4,-\alpha_3,-\alpha_1\}.
\end{align}

\begin{lemma}
	Consider the fundamental $\C^*$-action of $H_i$ on $(X,L)$, then $w_\circ P_j \in Y_+$. 
\end{lemma}

\begin{proof}
	Same proof as for Lemma \ref{lem:ePsource}, using $w_\circ \Phi^+=\Phi^-$.
\end{proof}

Combining the previous results of the section
, we obtain the following result.

\begin{theorem}
	Let $H_i$ acts on $(X,L)$ with weight map $\mu_L^i$, then 
	\[
	|\mu_L^i|=\sigma_i\left(\omega_j-w_\circ(\omega_j)\right).
	\]
\end{theorem}

\begin{proof}
	We recall that $\mu_L^i= \sigma_i \circ \mu_L^H$. 
	Thanks to the two previous lemmas, $\mu_L^i(Y_-)=\mu_L^i(eP_j)=0$ and $\mu_L^i(Y_+)=\mu_L^i(w_\circ P_j)$, hence by definition
	\[
	|\mu_L^i|=\mu_L^i(w_\circ P_j)-0=\sigma_i(\mu_L^H(w_\circ P_j))=\sigma_i(\omega_j-w_\circ (\omega_j)). \qedhere
	\] 
\end{proof}



At this point, to obtain the minimal bandwidth of Table \ref{tab:minBW}, it is enough to look at the tables in \cite[p. 250-275]{bourbaki6} for the decomposition of the fundamental weights in the basis of simple roots.
Moreover, we can generalize the previous result to the context of arbitrary RH variety, thanks to the linearity of $\sigma_i$.

\begin{corollary}
	Let $X=\cD(J)$ be a RH variety of arbitrary Picard number and let $L$ be an arbitrary line bundle corresponding to the dominant weight $\omega=\sum_{\alpha_j \in J} m_j\omega_j$. Suppose that the $\C^*$-action on $(X,L)$ correspond to $\sum_i p_i \sigma_i \in (\Z\Phi)^\vee$, then using linearization as in \eqref{eq:tuttoIntero}
	\[
	|\mu_L^i|=\sum_{\alpha_j \in J} m_j\sigma_i(\omega_j-w_\circ (\omega_j)).
	\]
\end{corollary}


\section{Projective description of classic cases}
\label{sec:examples}

We will use the numbering of nodes of Dynkin diagrams as in \cite{bourbaki6}.

\subsection{The case $\DA_n$: smooth drums}\label{ssec:An}

We start by $\DA_n(1)=\P^n$. 
We consider $H \subset \SL_{n+1}(\C)$, the maximal torus given by the diagonal matrices.
At this point we have to compute the polytope of fixed points $\Delta(\P^n):=\Delta(\P^n, \cO(1),H,\mu_{\cO(1)})$.
By \cite[Corollary 2.18]{BuWeWi} we know that
\[
\mu_{\cO(1)}(p')-\mu_{\cO(1)}(p)=\lambda \nu
\]
where $p,p'$ are fixed points (since $H$ is maximal, every fixed-point component is a fixed point by Proposition \ref{prop:gridsGP}) and $\nu \in \Comp(p,\P^n,H)$.
In particular, thanks to Proposition \ref{prop:gridsGP}, we have
\[
\Comp(eP_1,\P^n,H)=\left\{ \alpha \in \Phi^+: m_1(\alpha)>0\right\}=\left\{\alpha_1,\alpha_1+\alpha_2,\ldots,\alpha_1+\alpha_2+\ldots+\alpha_n\right\}.
\]
Using \cite{montagard2009regular}, we know that the polytope of fixed points is an $n$-simplex. 
Moreover, because $\mu_{\cO(1)}(eP_1)=0$, one can reconstruct all the other weights and they corresponds to the elements in the compass!

The $H$-actions on $\P^n$ is given by
\[
\left(t,[x_0:\ldots:x_n]\right) \mapsto \left[ x_0: t^{\alpha_1}x_1:\ldots: t^{\alpha_1+\ldots+\alpha_n}x_n \right],
\]
where $t^{\sum_i m_i \alpha_i}:=\prod_i t_i^{m_i \alpha_i}$ with $t_i \in \C^*$.
Finally, we choose the downgrading along $\alpha_1$:
\[
(t_1,[x_0:\ldots:x_n]) \mapsto \left[x_0: t_1^{1}x_1:\ldots:t_1^{1}x_n \right].
\]
This $\C^*$-action has two fixed-point components:
\begin{enumerate}
	\item the point $[e_0]:=[1:0:\ldots:0]$ with weight $0$,
	\item the hyperplane $\{x_0 = 0\}$ with weight $1$.
\end{enumerate}
Hence we recover that the $H_1$-action on $(\P^n, \cO(1))$ is $1$. 

We can repeat all this construction for the Grassmannians $\DA_n(k)$, obtaining again minimal bandwidth one. That is because these are examples of RH varieties which are \emph{smooth drums} in the sense of \cite[Section 4]{smallBandwidth}. In fact, \cite[Theorem 4.6]{smallBandwidth} characterizes the varieties $X$ which have a $\C^*$-action of bandwidth one. In the case $X$ is rational homogeneous, it has to be one of the smooth projective horospherical varieties of Picard number one described in \cite{pasquierPrinc}. In particular, the description of Grassmannians as horospherical varieties can be found in \cite[Proposition 1.9]{pasquierPrinc}.

\subsection{The case $\DD_n$}\label{ssec:Dn}

We start with an even dimensional quadric $\DD_n(1)$ given by $\cQ^{2n-2}=\left\{ x_0x_1+\ldots+x_{2n-2}x_{2n-1}=0\right\} \subset \P^{2n-1}$. 
Arguing as in the case $\DA_n$, the $\C^*$-action obtained by the downgrading along $\alpha_n$:
\[
\left(t_n,[x_0:\ldots:x_{2n-1}]\right) \mapsto \left[ x_0:t_n^{1}x_1:\ldots:x_{2n-2}:t_n^{1}x_{2n-1} \right].
\]
The fixed-point components are two $\P^{n-1}=\DA_{n-1}(1)$, as Theorem \ref{thm:Tevelev} and Proposition \ref{thm:trasvComp} predict. This is a bandwidth one $\C^*$-action 
(note that this is \cite[Proposition 1.8]{pasquierPrinc}).

If we compute the $\C^*$-action using the downgrading along $\alpha_1$, reduced to the action on $(\DD_n(1), \cO(1))$, the result is
\begin{equation}\label{eq:actionEvenQuadric}
	(t_1,[x_0:\ldots:x_{2n-1}]) \mapsto \left[ t_1^{-1}x_0: t_1x_1 : x_2 : \ldots: x_{2n-1} \right].
\end{equation}
We have three fixed-point components this time:
\begin{enumerate}
	\item the sink $Y_-=[e_0]$ with weight $-1$,
	\item a smaller quadric, $\DD_{n-1}(1)=\cQ^{2n-4}=\cQ^{2n-2} \cap \{x_0=x_1=0\}$, with weight $0$,
	\item the source $Y_+=[e_1]$ with weight $1$.
\end{enumerate}

In order to produce a bandwidth two action on $\DD_n(k)$ for $k <n-1$, which has to be minimal by Theorem \ref{thm:main}, we consider the action induced by \eqref{eq:actionEvenQuadric} on the isotropic Grassmannians.
Because the line spanned by $Y_+$ and $Y_+$ is not contained in the quadric, the invariant lines are:
\begin{enumerate}
\item any line spanned by $Y_+$ and a point of $\cQ^{2n-4}$, forming an even dimensional quadric $\DD_{n-1}(1)$ with weight $1$;
\item any line spanned by $Y_-$ and a point of $\cQ^{2n-4}$, which is $\DD_{n-1}(1)$ with weight $-1$;
\item any line contained in $\DD_{n-1}(1)$; the variety of these lines is $\DD_{n-1}(2)$ on which we have weight $0$.
\end{enumerate}
We have summarized the situation in Figure \ref{fig:Dn2}.
\begin{figure}[ht!]
\centering
\tikzset{every picture/.style={line width=0.75pt}} 

\begin{tikzpicture}[x=0.5pt,y=0.5pt,yscale=-1,xscale=1]

\draw   (190,110) .. controls (190,82.39) and (243.73,60) .. (310,60) .. controls (376.27,60) and (430,82.39) .. (430,110) .. controls (430,137.61) and (376.27,160) .. (310,160) .. controls (243.73,160) and (190,137.61) .. (190,110) -- cycle ;
\draw    (70,259) -- (550,259) ;
\draw    (110,249) -- (110,269) ;
\draw    (310,249) -- (310,269) ;
\draw    (510,249) -- (510,269) ;
\draw [line width=1.5]    (110,110) -- (270,140) ;
\draw [line width=1.5]    (340,80) -- (510,110) ;
\draw [line width=1.5]    (270,80) -- (360,120) ;
\draw [line width=1.5]    (250,90) -- (330,140) ;
\draw    (170,130) -- (111.2,208.4) ;
\draw [shift={(110,210)}, rotate = 306.87] [color={rgb, 255:red, 0; green, 0; blue, 0 }  ][line width=0.75]    (10.93,-3.29) .. controls (6.95,-1.4) and (3.31,-0.3) .. (0,0) .. controls (3.31,0.3) and (6.95,1.4) .. (10.93,3.29)   ;
\draw    (310,140) -- (310,208) ;
\draw [shift={(310,210)}, rotate = 270] [color={rgb, 255:red, 0; green, 0; blue, 0 }  ][line width=0.75]    (10.93,-3.29) .. controls (6.95,-1.4) and (3.31,-0.3) .. (0,0) .. controls (3.31,0.3) and (6.95,1.4) .. (10.93,3.29)   ;
\draw    (450,120) -- (508.89,208.34) ;
\draw [shift={(510,210)}, rotate = 236.31] [color={rgb, 255:red, 0; green, 0; blue, 0 }  ][line width=0.75]    (10.93,-3.29) .. controls (6.95,-1.4) and (3.31,-0.3) .. (0,0) .. controls (3.31,0.3) and (6.95,1.4) .. (10.93,3.29)   ;
\draw [line width=1.5]    (110,110) -- (230,100) ;
\draw [line width=1.5]    (380,110) -- (510,110) ;
\draw  [fill={rgb, 255:red, 0; green, 0; blue, 0 }  ,fill opacity=1 ] (105,110) .. controls (105,107.24) and (107.24,105) .. (110,105) .. controls (112.76,105) and (115,107.24) .. (115,110) .. controls (115,112.76) and (112.76,115) .. (110,115) .. controls (107.24,115) and (105,112.76) .. (105,110) -- cycle ;
\draw  [fill={rgb, 255:red, 0; green, 0; blue, 0 }  ,fill opacity=1 ] (505,110) .. controls (505,107.24) and (507.24,105) .. (510,105) .. controls (512.76,105) and (515,107.24) .. (515,110) .. controls (515,112.76) and (512.76,115) .. (510,115) .. controls (507.24,115) and (505,112.76) .. (505,110) -- cycle ;

\draw (96,272) node [anchor=north west][inner sep=0.75pt]    {$-1$};
\draw (305,272) node [anchor=north west][inner sep=0.75pt]    {$0$};
\draw (505,272) node [anchor=north west][inner sep=0.75pt]    {$1$};
\draw (100,73) node [anchor=north west][inner sep=0.75pt]    {$Y_-$};
\draw (500,73) node [anchor=north west][inner sep=0.75pt]    {$Y_+$};
\draw (295,23) node [anchor=north west][inner sep=0.75pt]    {$\cQ^{2n-4}$};
\draw (85,220) node [anchor=north west][inner sep=0.75pt]    {$\DD_{n-1}( 1)$};
\draw (285,220) node [anchor=north west][inner sep=0.75pt]    {$\DD_{n-1}( 2)$};
\draw (485,220) node [anchor=north west][inner sep=0.75pt]    {$\DD_{n-1}( 1)$};

\end{tikzpicture}
\caption{The induced action from the downgrading along $\alpha_1$ on $\DD_n(2)$.}
\label{fig:Dn2}
\end{figure}
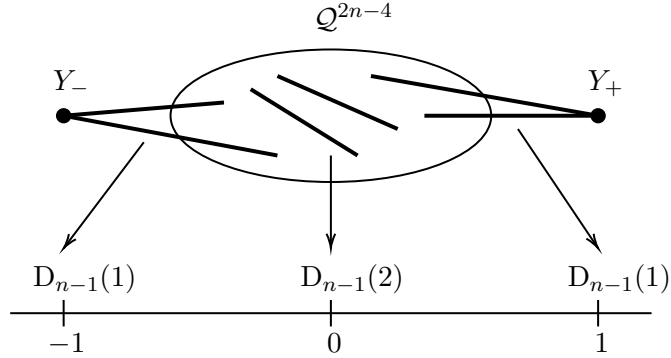

This argument can be generalized for $\mathrm{D}_n(k)$ with $k<n-1$.

Last, we look at the induced action on $\DD_n(n)$ (the case of $\DD_n(n-1)$ is analogous). It is a well-known fact that there are two families of $\P^{n-1}$'s contained in $\cQ^{2n-2}$, that we denote by $\P^{n-1}_a$'s and $\P^{n-1}_b$'s. We suppose that $\DD_n(n)$ parametrize the family of $\P^{n-1}_a$'s.
At this point, a fixed-point $\P^{n-1}_a$ in $\DD_n(n)$ can be only if two types:
\begin{enumerate}
	\item spanned by $Y_-$ and a $\P^{n-2}_a$ contained in $\cQ^{2n-4}$, denoted by $Y_{-1}$,
	\item spanned by $Y_+$ and a $\P^{n-2}_a$ contained in $\cQ^{2n-4}$, $Y_1$.
\end{enumerate}
We expect that this can be reduced to a bandwidth one action, since
\[
\omega_n=\frac{1}{2}\left( \alpha_1+2\alpha_2+\ldots+(n-2)\alpha_{n-2} +\frac{1}{2}(n-1) \alpha_{n-1}+\frac{1}{2}n\alpha_n\right).
\] 
In fact, this is exactly \cite[Proposition 1.10]{pasquierPrinc}.

\subsection{The case $\DB_n$}\label{ssec:caseBn}

This case is very similar to the previous one, so we describe only the differences. The odd quadric $\DB_n(1)=\cQ^{2n-1}=\left\{ x_0x_1+\ldots+x_{2n-2}x_{2n-1} + x_{2n}^2=0\right\}$ is the closed orbit into the projectivization of the representation given by
\[
\omega_1=\alpha_1+\ldots+\alpha_n.
\]
The main difference with the $\DD_n$-case is that, although we have the same polytope of fixed points (an $n$-orthoplex), the action given by the downgrading along $\alpha_n$ is given by
\begin{equation} \label{nonequalizedBn}
	\left(t_n,[x_0:\ldots:x_{2n}]\right) \mapsto \left[ t_n^{-1}x_0:t_nx_1:\ldots:t_n^{-1}x_{2n-2}:t_nx_{2n-1}: x_{2n} \right].
\end{equation}
Hence the $x_{2n}$-coordinate does not allow us to construct the bandwidth one action, contrary to the case of the even quadric $\cQ^{2n-2}=\cQ^{2n-1} \cap \{ x_{2n}=0\}$.

\begin{remark}\label{rem:BnNotEqualized}
	The $\C^*$-action \eqref{nonequalizedBn} on $\DB_n(1)$ is not equalized. Thanks to Lemma \ref{lem:AMvsFM2} we have that
	\[
		2=\delta(C) \deg f^*\cO(1)
	\]
	where $C$ is the closure of a general orbit. Suppose that the closure of such an orbit is a line. Then it is a line passing by the points  $p_1:=[x_0:0:x_2:0:\ldots:x_{2n-2}:0:0]$ and $p_2:=[0:x_1:0:x_3:\ldots:0:x_{2n-1}:0]$ which has the form
	\[
		\overline{\left\{[x_0:sx_1:\ldots:x_{2n-2}:sx_{2n-1}:0]: s \in \C^*\right\}}.
	\]
	But, for a general choice of $p_1$ and $p_2$, the points in the line do not satisfy the equation of $\cQ^{2n-1}$. Hence the general orbit is at least a conic, hence $\deg f^*\cO(1)\ge 2$ and $\delta(C) \le 1$. But $\delta(C)$ is an integer strictly bigger than zero, so $\delta(C)=1$ and $\deg f^*\cO(1)=2$.
	Now consider $q_1:=[1:0:\ldots:0]$ and $q_2:=[0:\ldots:0:1:0]$, the line passing by them is 
	\[
		\tilde{C}=\overline{\left\{[1:0:\ldots:0:s:0] : s \in \C^*\right\}} \subset \cQ^{2n-1}.
	\]
	So $\delta(\tilde{C})=2$ for this line and the action is not equalized by Remark \ref{rem:equalized}.
\end{remark}

On the other hand, the downgrading along $\alpha_1$ produces the same action of the $\DD_n$-case, so we skip the computations. One can also compute the minimal bandwidth of $\DB_n(n)$ using the isomorphism $\DB_n(n) \simeq \DD_{n+1}(n+1)$.

\subsection{The case $\DC_n$}\label{ssec:caseCn}

The isomorphism $\DC_n(1) \simeq \DA_{2n+1}(1)= \P^{2n-1}$ gives us a bandwidth one action. This is given by the downgrading along $\alpha_n$, in fact
\[
	\omega_1=\alpha_1+\ldots+\alpha_{n-1}+\frac{1}{2}\alpha_n.
\]
If one consider the $\C^*$-action given by $\alpha_1$, this is exactly the same action as in the case of the quadrics:
\[
	\left(t_1,[x_0:\ldots:x_{2n}]\right) \mapsto \left[ t_1^{-1}x_0:t_1x_1:x_2:\ldots:x_{2n-1} \right].
\]
Hence the induced action on $\DC_n(k)$ gives us a bandwidth two action, where the fixed-point components are:
\begin{enumerate}
	\item a $\DC_{n-1}(k-1)$ spanned by $Y_-$ and a $\P^{k-2} \subset \DC_{n-1}(1)=\P^{2n-3}$, with weight $-1$,
	\item the isotropic Grassmannian $\DC_{n-1}(k)$ with weight $0$,
	\item again a $\DC_{n-1}(k-1)$ spanned by $Y_+$ and a $\P^{k-2} \subset \DC_{n-1}(1)$ with weight $1$. 
\end{enumerate}
There are no other fixed-point components because, if $\Omega$ is the symplectic form defining $\SP(2n)$, then an isotropic line is spanned by $[a]$ and $[b]$ such that $\Omega(a,b)=0$ and $\Omega(Y_-,Y_+)\ne 0$. In particular, there is no isotropic linear space of any dimension passing through $Y_+$ and $Y_-$.

We conclude with a Remark on the $\C^*$-action given by $\alpha_1$ along $\DC_n(n)$.

\begin{remark}\label{rem:CnNotEqualized}
The $\C^*$-action on $\DC_n(n)$ given by the downgrading along $\alpha_1$ is not equalized.
Fix an isotropic $\P^{n-2} \in \DC_{n-1}(1)$. Then consider the $\P^{n}$ generated by the fixed-point $\P^{n-2}$, by $Y_+$ and $Y_-$. Inside this $\P^{n}$ there is a pencil of isotropic $\P^{n-1}$'s, i.e. a line between the fixed-point components of $\DC_n(n)$: consider the line between $Y_+$ and $Y_-$ given by
\[
	C=\overline{\{[\underbrace{1:0:\ldots:0}_n:\underbrace{s:0:\ldots:0}_{n}] \in \P^{2n-1} : s \in \C^*\}},
\]
then the $\P^{n-1}$'s are generated by the fixed-point $\P^{n-2}$'s and by a point in $C$.

Then, by Lemma \ref{lem:AMvsFM2} it must have $2=\delta(C) \deg f^*\cO(1)$ and $\deg f^*\cO(1)=1$ by the previous argument. So we have found the closure of a 1-dimensional orbit for which $\delta(C)=2$, hence the action is not equalized.

In the case $n=2$, this coincide with Remark \ref{rem:BnNotEqualized}.
\end{remark}


\section{Chow groups of the complex Cayley plane}\label{sec:Cayley}

The Cayley plane $X$ can be described in various ways: it is the fourth Severi variety of dimension $16$, hence the results of \cite{ManivelFreud} tell us that $X=\P^2(\O)$, where $\O=\mathbf{O} \otimes_\R \C$ ($\mathbf{O}$ is the algebra of octonions).
We prefer to describe $X$ as an RH variety, hence $X:=\DE_6(6)$.

\begin{remark}
	Thanks to Theorem \ref{thm:main}, the minimal bandwidth of $X$ is $2$ and a fundamental $\C^*$-action that attains this minimum is given by the downgrading along $\alpha_6$ (using the numbering of \cite{bourbaki6}). Moreover, using Theorem \ref{thm:Tevelev}, the fixed-point components $Y_0,Y_1,Y_2$ are of type $\DD_5$.
\end{remark}

\begin{proposition}
	Consider the fundamental $\C^*$-action of $H_6$ on $X \subset \P(V(\omega
	_6))$. Then
	\[
	X^{H_6}=\{*\} \sqcup \DD_5(5) \sqcup \DD_5(1).
	\]
	Moreover, the links between fixed-point components are given by
	\[
	\xymatrix{ & \DD_5(5) \ar[ld]_{\DD_5(5) \subset \P^{15}} \ar[rd]^{\P^0} & & \DD_5(1,5) \ar[ld]_{\P^4} \ar[rd]^{\DD_4(1) \subset \P^{7}} & \\
	\{*\} & & \DD_5(5) & & \DD_5(1)
	}.
	\]
\end{proposition}

\begin{proof}
	The fundamental representation $V(\omega_6)$ has dimension $27$, hence $X \subset \P^{26}$ and the action of the maximal torus $H$ on $X$ has at most $27$ fixed points.
	
	Using Lemma \ref{lem:ePsource}, we conclude that $Y_0$ is a point. Moreover, the compass for $H$ at $Y_2$ is given by
	\[
	\Comp(Y_2,X,H)=\{\alpha \in \Phi^+(\DE_6,H) : \alpha-\alpha_6 \ge 0\}
	\]
	thanks to Proposition \ref{prop:gridsGP}. 
	
	Let us note that there are exactly $16$ elements in this compass, one for every fixed point for $H$ which is sent to $Y_1$ by the downgrading along $\alpha_6$. Hence $Y_1 \subset \P^{15}$ is an RH variety of type $\DD_5$. The only variety with these properties is the $10$-dimensional spinor variety $\DD_5(5)$ (cf. \cite[Proposition 4.2]{ManivelCayleyPlane}). Last, there are $6$ elements in the compass for the action of $H$ at $Y_1$. Every point of $Y_1$ is linked to the fixed point $Y_2$ by a single direction in the compass, hence
	\[
	\nu_+(Y_1):=|\{ \nu_i(Y_1) >0\}|=1 \quad \text{and} \quad \nu_-(Y_1):=|\{\nu_i(Y_1) <0\}|=5.
	\]
	This information allow us to recover, respectively, the $\P^0$ and the $\P^4$ on the diagram.
	
	Finally, We know that $Y_1$ has to parametrize a family of $\P^4$ of $Y_2$, which is RH of type $\DD_5$. It follows that $Y_0=\DD_5(1)$. Last, there are $8$ (negative) elements in the compass at $Y_2$, and the incidence diagram tells us that every point in $Y_2$ is linked to a $6$-dimensional quadric $\DD_4(1) \subset \P^7$, as we expect.
\end{proof}

We summarize the information about positive and negative elements in the compass at the fixed-point components in the following table.
\begin{table}[ht!]
\centering
\begin{tabular}{c|c|c|c}
 & $Y_0$ & $Y_1$ & $Y_2$ \\
 \hline
 $\nu_+(Y_i)$ & $16$ & $5$ & $0$ \\
 \hline
 $\nu_-(Y_i)$ & $0$ & $1$ & $8$
\end{tabular}
\end{table}

Using this table we can rewrite the positive decomposition \eqref{eq:decoHomBB} as
\begin{equation}\label{eq:decoE6(1)}
	H_\bullet (X,\Z)=H_{\bullet}(\DD_5(1),\Z) \oplus H_{\bullet-10}(\DD_5(5),\Z) \oplus H_{\bullet-32}(\{*\},\Z).
\end{equation}
We are in the hypothesis of the Poincaré duality, hence we have 
\[
H_m(Y_i,\Z) \simeq H^{2\dim(Y_i)-m}(Y_i,\Z),
\]
and we obtain the additive decomposition of the cohomology ring as
\begin{equation}
H^\bullet(X,\Z)=H^{\bullet-16}(\DD_5(1),\Z) \oplus H^{\bullet-2}(\DD_5(5),\Z) \oplus H^{\bullet}(\{*\},\Z).	
\end{equation}
Finally we can decompose additively the Chow groups of $X$ as
\begin{equation}
	A^\bullet(X)=A^{\bullet-8}(\DD_5(1)) \oplus A^{\bullet-1}(\DD_5(5)) \oplus A^\bullet(\{*\}).
\end{equation}

\begin{remark}
	As pointed out in \cite[Section 3]{ManivelCayleyPlane}, one can study the Chow ring of $X$ using the corresponding Hasse diagram. In particular, Iliev and Manivel describe a Schubert subvariety of $X$ isomorphic to $\DD_5(5)$, finding a Hasse sub-diagram that corresponds to the one associated to $\DD_5(5)$. 
	
Our description in terms of $\C^*$-action on $X$ allows us to interpret the complement of the Hasse sub-diagram of $\DD_5(5)$ in $X$ as the union of two sub-diagrams, corresponding to the Chow groups of the sink (a point) and the source ($\DD_5(1)$) of the $\C^*$-action.

By the way, we cannot recover the ring structure of $A^\bullet(X)$, because our methods are based essentially on the Bia\l ynicki-Birula theorem. This requires a case-by-case analysis, as, for example, Iliev and Manivel did in \cite{ManivelCayleyPlane}. For sure, the knowledge of such fixed-point decomposition can help the analysis.

	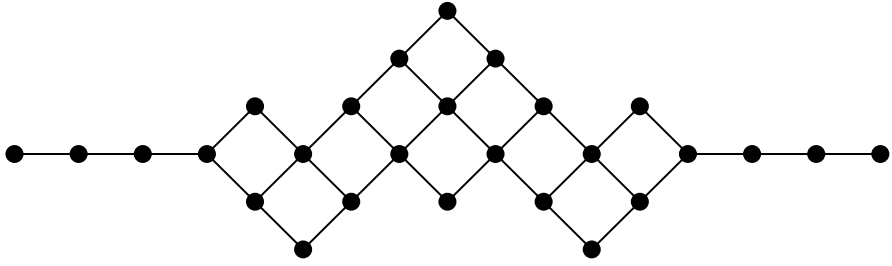
\begin{figure}
		\centering
		\tikzset{every picture/.style={line width=0.75pt}} 

\begin{tikzpicture}[x=0.6pt,y=0.6pt,yscale=-1,xscale=1]

\draw    (20,110) -- (60,110) ;
\draw    (60,110) -- (100,110) ;
\draw  [fill={rgb, 255:red, 0; green, 0; blue, 0 }  ,fill opacity=1 ] (15,110) .. controls (15,107.24) and (17.24,105) .. (20,105) .. controls (22.76,105) and (25,107.24) .. (25,110) .. controls (25,112.76) and (22.76,115) .. (20,115) .. controls (17.24,115) and (15,112.76) .. (15,110) -- cycle ;
\draw    (100,110) -- (140,110) ;
\draw    (140,110) -- (170,140) ;
\draw    (140,110) -- (170,80) ;
\draw    (170,80) -- (200,110) ;
\draw    (200,110) -- (230,140) ;
\draw    (170,140) -- (200,170) ;
\draw    (170,140) -- (200,110) ;
\draw    (200,170) -- (230,140) ;
\draw    (200,110) -- (230,80) ;
\draw    (230,80) -- (260,50) ;
\draw    (260,50) -- (290,20) ;
\draw    (230,140) -- (260,110) ;
\draw    (260,110) -- (290,80) ;
\draw    (290,80) -- (320,50) ;
\draw    (230,80) -- (260,110) ;
\draw    (260,50) -- (290,80) ;
\draw    (290,20) -- (320,50) ;
\draw    (260,110) -- (290,140) ;
\draw    (290,80) -- (320,110) ;
\draw    (320,50) -- (350,80) ;
\draw    (290,140) -- (320,110) ;
\draw    (320,110) -- (350,80) ;
\draw    (320,110) -- (350,140) ;
\draw    (350,80) -- (380,110) ;
\draw    (350,140) -- (380,170) ;
\draw    (380,110) -- (410,140) ;
\draw    (350,140) -- (380,110) ;
\draw    (380,170) -- (410,140) ;
\draw    (380,110) -- (410,80) ;
\draw    (410,140) -- (440,110) ;
\draw    (410,80) -- (440,110) ;
\draw    (440,110) -- (480,110) ;
\draw    (480,110) -- (520,110) ;
\draw    (520,110) -- (560,110) ;
\draw  [fill={rgb, 255:red, 0; green, 0; blue, 0 }  ,fill opacity=1 ] (55,110) .. controls (55,107.24) and (57.24,105) .. (60,105) .. controls (62.76,105) and (65,107.24) .. (65,110) .. controls (65,112.76) and (62.76,115) .. (60,115) .. controls (57.24,115) and (55,112.76) .. (55,110) -- cycle ;
\draw  [fill={rgb, 255:red, 0; green, 0; blue, 0 }  ,fill opacity=1 ] (95,110) .. controls (95,107.24) and (97.24,105) .. (100,105) .. controls (102.76,105) and (105,107.24) .. (105,110) .. controls (105,112.76) and (102.76,115) .. (100,115) .. controls (97.24,115) and (95,112.76) .. (95,110) -- cycle ;
\draw  [fill={rgb, 255:red, 0; green, 0; blue, 0 }  ,fill opacity=1 ] (135,110) .. controls (135,107.24) and (137.24,105) .. (140,105) .. controls (142.76,105) and (145,107.24) .. (145,110) .. controls (145,112.76) and (142.76,115) .. (140,115) .. controls (137.24,115) and (135,112.76) .. (135,110) -- cycle ;
\draw  [fill={rgb, 255:red, 0; green, 0; blue, 0 }  ,fill opacity=1 ] (165,80) .. controls (165,77.24) and (167.24,75) .. (170,75) .. controls (172.76,75) and (175,77.24) .. (175,80) .. controls (175,82.76) and (172.76,85) .. (170,85) .. controls (167.24,85) and (165,82.76) .. (165,80) -- cycle ;
\draw  [fill={rgb, 255:red, 0; green, 0; blue, 0 }  ,fill opacity=1 ] (165,140) .. controls (165,137.24) and (167.24,135) .. (170,135) .. controls (172.76,135) and (175,137.24) .. (175,140) .. controls (175,142.76) and (172.76,145) .. (170,145) .. controls (167.24,145) and (165,142.76) .. (165,140) -- cycle ;
\draw  [fill={rgb, 255:red, 0; green, 0; blue, 0 }  ,fill opacity=1 ] (195,170) .. controls (195,167.24) and (197.24,165) .. (200,165) .. controls (202.76,165) and (205,167.24) .. (205,170) .. controls (205,172.76) and (202.76,175) .. (200,175) .. controls (197.24,175) and (195,172.76) .. (195,170) -- cycle ;
\draw  [fill={rgb, 255:red, 0; green, 0; blue, 0 }  ,fill opacity=1 ] (195,110) .. controls (195,107.24) and (197.24,105) .. (200,105) .. controls (202.76,105) and (205,107.24) .. (205,110) .. controls (205,112.76) and (202.76,115) .. (200,115) .. controls (197.24,115) and (195,112.76) .. (195,110) -- cycle ;
\draw  [fill={rgb, 255:red, 0; green, 0; blue, 0 }  ,fill opacity=1 ] (225,140) .. controls (225,137.24) and (227.24,135) .. (230,135) .. controls (232.76,135) and (235,137.24) .. (235,140) .. controls (235,142.76) and (232.76,145) .. (230,145) .. controls (227.24,145) and (225,142.76) .. (225,140) -- cycle ;
\draw  [fill={rgb, 255:red, 0; green, 0; blue, 0 }  ,fill opacity=1 ] (225,80) .. controls (225,77.24) and (227.24,75) .. (230,75) .. controls (232.76,75) and (235,77.24) .. (235,80) .. controls (235,82.76) and (232.76,85) .. (230,85) .. controls (227.24,85) and (225,82.76) .. (225,80) -- cycle ;
\draw  [fill={rgb, 255:red, 0; green, 0; blue, 0 }  ,fill opacity=1 ] (255,110) .. controls (255,107.24) and (257.24,105) .. (260,105) .. controls (262.76,105) and (265,107.24) .. (265,110) .. controls (265,112.76) and (262.76,115) .. (260,115) .. controls (257.24,115) and (255,112.76) .. (255,110) -- cycle ;
\draw  [fill={rgb, 255:red, 0; green, 0; blue, 0 }  ,fill opacity=1 ] (255,50) .. controls (255,47.24) and (257.24,45) .. (260,45) .. controls (262.76,45) and (265,47.24) .. (265,50) .. controls (265,52.76) and (262.76,55) .. (260,55) .. controls (257.24,55) and (255,52.76) .. (255,50) -- cycle ;
\draw  [fill={rgb, 255:red, 0; green, 0; blue, 0 }  ,fill opacity=1 ] (285,20) .. controls (285,17.24) and (287.24,15) .. (290,15) .. controls (292.76,15) and (295,17.24) .. (295,20) .. controls (295,22.76) and (292.76,25) .. (290,25) .. controls (287.24,25) and (285,22.76) .. (285,20) -- cycle ;
\draw  [fill={rgb, 255:red, 0; green, 0; blue, 0 }  ,fill opacity=1 ] (285,80) .. controls (285,77.24) and (287.24,75) .. (290,75) .. controls (292.76,75) and (295,77.24) .. (295,80) .. controls (295,82.76) and (292.76,85) .. (290,85) .. controls (287.24,85) and (285,82.76) .. (285,80) -- cycle ;
\draw  [fill={rgb, 255:red, 0; green, 0; blue, 0 }  ,fill opacity=1 ] (285,140) .. controls (285,137.24) and (287.24,135) .. (290,135) .. controls (292.76,135) and (295,137.24) .. (295,140) .. controls (295,142.76) and (292.76,145) .. (290,145) .. controls (287.24,145) and (285,142.76) .. (285,140) -- cycle ;
\draw  [fill={rgb, 255:red, 0; green, 0; blue, 0 }  ,fill opacity=1 ] (315,110) .. controls (315,107.24) and (317.24,105) .. (320,105) .. controls (322.76,105) and (325,107.24) .. (325,110) .. controls (325,112.76) and (322.76,115) .. (320,115) .. controls (317.24,115) and (315,112.76) .. (315,110) -- cycle ;
\draw  [fill={rgb, 255:red, 0; green, 0; blue, 0 }  ,fill opacity=1 ] (315,50) .. controls (315,47.24) and (317.24,45) .. (320,45) .. controls (322.76,45) and (325,47.24) .. (325,50) .. controls (325,52.76) and (322.76,55) .. (320,55) .. controls (317.24,55) and (315,52.76) .. (315,50) -- cycle ;
\draw  [fill={rgb, 255:red, 0; green, 0; blue, 0 }  ,fill opacity=1 ] (345,80) .. controls (345,77.24) and (347.24,75) .. (350,75) .. controls (352.76,75) and (355,77.24) .. (355,80) .. controls (355,82.76) and (352.76,85) .. (350,85) .. controls (347.24,85) and (345,82.76) .. (345,80) -- cycle ;
\draw  [fill={rgb, 255:red, 0; green, 0; blue, 0 }  ,fill opacity=1 ] (375,110) .. controls (375,107.24) and (377.24,105) .. (380,105) .. controls (382.76,105) and (385,107.24) .. (385,110) .. controls (385,112.76) and (382.76,115) .. (380,115) .. controls (377.24,115) and (375,112.76) .. (375,110) -- cycle ;
\draw  [fill={rgb, 255:red, 0; green, 0; blue, 0 }  ,fill opacity=1 ] (345,140) .. controls (345,137.24) and (347.24,135) .. (350,135) .. controls (352.76,135) and (355,137.24) .. (355,140) .. controls (355,142.76) and (352.76,145) .. (350,145) .. controls (347.24,145) and (345,142.76) .. (345,140) -- cycle ;
\draw  [fill={rgb, 255:red, 0; green, 0; blue, 0 }  ,fill opacity=1 ] (375,170) .. controls (375,167.24) and (377.24,165) .. (380,165) .. controls (382.76,165) and (385,167.24) .. (385,170) .. controls (385,172.76) and (382.76,175) .. (380,175) .. controls (377.24,175) and (375,172.76) .. (375,170) -- cycle ;
\draw  [fill={rgb, 255:red, 0; green, 0; blue, 0 }  ,fill opacity=1 ] (405,140) .. controls (405,137.24) and (407.24,135) .. (410,135) .. controls (412.76,135) and (415,137.24) .. (415,140) .. controls (415,142.76) and (412.76,145) .. (410,145) .. controls (407.24,145) and (405,142.76) .. (405,140) -- cycle ;
\draw  [fill={rgb, 255:red, 0; green, 0; blue, 0 }  ,fill opacity=1 ] (405,80) .. controls (405,77.24) and (407.24,75) .. (410,75) .. controls (412.76,75) and (415,77.24) .. (415,80) .. controls (415,82.76) and (412.76,85) .. (410,85) .. controls (407.24,85) and (405,82.76) .. (405,80) -- cycle ;
\draw  [fill={rgb, 255:red, 0; green, 0; blue, 0 }  ,fill opacity=1 ] (435,110) .. controls (435,107.24) and (437.24,105) .. (440,105) .. controls (442.76,105) and (445,107.24) .. (445,110) .. controls (445,112.76) and (442.76,115) .. (440,115) .. controls (437.24,115) and (435,112.76) .. (435,110) -- cycle ;
\draw  [fill={rgb, 255:red, 0; green, 0; blue, 0 }  ,fill opacity=1 ] (475,110) .. controls (475,107.24) and (477.24,105) .. (480,105) .. controls (482.76,105) and (485,107.24) .. (485,110) .. controls (485,112.76) and (482.76,115) .. (480,115) .. controls (477.24,115) and (475,112.76) .. (475,110) -- cycle ;
\draw  [fill={rgb, 255:red, 0; green, 0; blue, 0 }  ,fill opacity=1 ] (515,110) .. controls (515,107.24) and (517.24,105) .. (520,105) .. controls (522.76,105) and (525,107.24) .. (525,110) .. controls (525,112.76) and (522.76,115) .. (520,115) .. controls (517.24,115) and (515,112.76) .. (515,110) -- cycle ;
\draw  [fill={rgb, 255:red, 0; green, 0; blue, 0 }  ,fill opacity=1 ] (555,110) .. controls (555,107.24) and (557.24,105) .. (560,105) .. controls (562.76,105) and (565,107.24) .. (565,110) .. controls (565,112.76) and (562.76,115) .. (560,115) .. controls (557.24,115) and (555,112.76) .. (555,110) -- cycle ;

\end{tikzpicture}
\caption{Hasse diagram of the Chow ring of $X$}
	\end{figure}
	
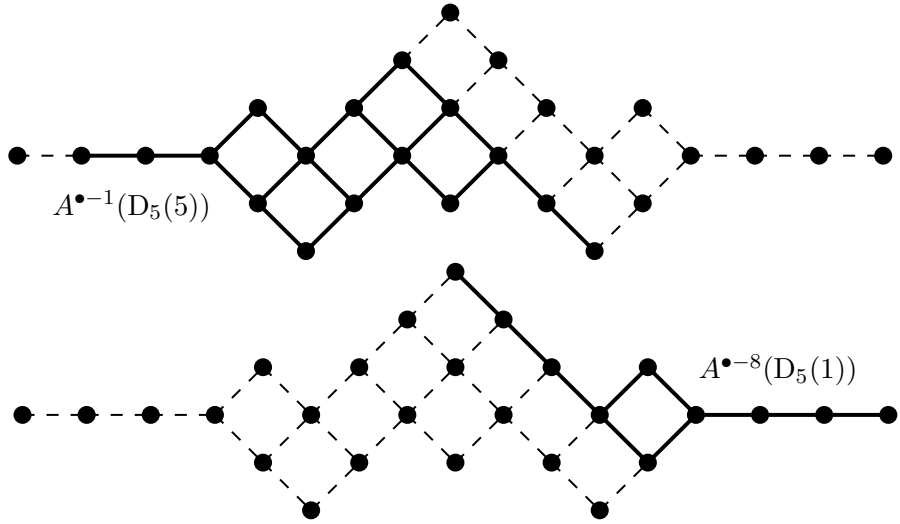
\begin{figure}

\tikzset{every picture/.style={line width=0.75pt}} 

\begin{tikzpicture}[x=0.6pt,y=0.6pt,yscale=-1,xscale=1]

\draw  [dash pattern={on 4.5pt off 4.5pt}]  (20,110) -- (60,110) ;
\draw [line width=1.50]    (60,110) -- (100,110) ;
\draw  [fill={rgb, 255:red, 0; green, 0; blue, 0 }  ,fill opacity=1 ] (15,110) .. controls (15,107.24) and (17.24,105) .. (20,105) .. controls (22.76,105) and (25,107.24) .. (25,110) .. controls (25,112.76) and (22.76,115) .. (20,115) .. controls (17.24,115) and (15,112.76) .. (15,110) -- cycle ;
\draw [line width=1.50]    (100,110) -- (140,110) ;
\draw [line width=1.50]    (140,110) -- (170,140) ;
\draw [line width=1.50]    (140,110) -- (170,80) ;
\draw [line width=1.50]    (170,80) -- (200,110) ;
\draw [line width=1.50]    (200,110) -- (230,140) ;
\draw [line width=1.50]    (170,140) -- (200,170) ;
\draw [line width=1.50]    (170,140) -- (200,110) ;
\draw [line width=1.50]    (200,170) -- (230,140) ;
\draw [line width=1.50]    (200,110) -- (230,80) ;
\draw [line width=1.50]    (230,80) -- (260,50) ;
\draw  [dash pattern={on 4.5pt off 4.5pt}]  (260,50) -- (290,20) ;
\draw [line width=1.50]    (230,140) -- (260,110) ;
\draw [line width=1.50]    (260,110) -- (290,80) ;
\draw  [dash pattern={on 4.5pt off 4.5pt}]  (290,80) -- (320,50) ;
\draw [line width=1.50]    (230,80) -- (260,110) ;
\draw [line width=1.50]    (260,50) -- (290,80) ;
\draw  [dash pattern={on 4.5pt off 4.5pt}]  (290,20) -- (320,50) ;
\draw [line width=1.50]    (260,110) -- (290,140) ;
\draw [line width=1.50]    (290,80) -- (320,110) ;
\draw  [dash pattern={on 4.5pt off 4.5pt}]  (320,50) -- (350,80) ;
\draw [line width=1.50]    (290,140) -- (320,110) ;
\draw  [dash pattern={on 4.5pt off 4.5pt}]  (320,110) -- (350,80) ;
\draw [line width=1.50]    (320,110) -- (350,140) ;
\draw  [dash pattern={on 4.5pt off 4.5pt}]  (350,80) -- (380,110) ;
\draw [line width=1.50]    (350,140) -- (380,170) ;
\draw  [dash pattern={on 4.5pt off 4.5pt}]  (380,110) -- (410,140) ;
\draw  [dash pattern={on 4.5pt off 4.5pt}]  (350,140) -- (380,110) ;
\draw  [dash pattern={on 4.5pt off 4.5pt}]  (380,170) -- (410,140) ;
\draw  [dash pattern={on 4.5pt off 4.5pt}]  (380,110) -- (410,80) ;
\draw  [dash pattern={on 4.5pt off 4.5pt}]  (410,140) -- (440,110) ;
\draw  [dash pattern={on 4.5pt off 4.5pt}]  (410,80) -- (440,110) ;
\draw  [dash pattern={on 4.5pt off 4.5pt}]  (440,110) -- (480,110) ;
\draw  [dash pattern={on 4.5pt off 4.5pt}]  (480,110) -- (520,110) ;
\draw  [dash pattern={on 4.5pt off 4.5pt}]  (520,110) -- (560,110) ;
\draw  [fill={rgb, 255:red, 0; green, 0; blue, 0 }  ,fill opacity=1 ] (55,110) .. controls (55,107.24) and (57.24,105) .. (60,105) .. controls (62.76,105) and (65,107.24) .. (65,110) .. controls (65,112.76) and (62.76,115) .. (60,115) .. controls (57.24,115) and (55,112.76) .. (55,110) -- cycle ;
\draw  [fill={rgb, 255:red, 0; green, 0; blue, 0 }  ,fill opacity=1 ]  (95,110) .. controls (95,107.24) and (97.24,105) .. (100,105) .. controls (102.76,105) and (105,107.24) .. (105,110) .. controls (105,112.76) and (102.76,115) .. (100,115) .. controls (97.24,115) and (95,112.76) .. (95,110) -- cycle ;
\draw  [fill={rgb, 255:red, 0; green, 0; blue, 0 }  ,fill opacity=1 ]  (135,110) .. controls (135,107.24) and (137.24,105) .. (140,105) .. controls (142.76,105) and (145,107.24) .. (145,110) .. controls (145,112.76) and (142.76,115) .. (140,115) .. controls (137.24,115) and (135,112.76) .. (135,110) -- cycle ;
\draw  [fill={rgb, 255:red, 0; green, 0; blue, 0 }  ,fill opacity=1 ]  (165,80) .. controls (165,77.24) and (167.24,75) .. (170,75) .. controls (172.76,75) and (175,77.24) .. (175,80) .. controls (175,82.76) and (172.76,85) .. (170,85) .. controls (167.24,85) and (165,82.76) .. (165,80) -- cycle ;
\draw  [fill={rgb, 255:red, 0; green, 0; blue, 0 }  ,fill opacity=1 ]  (165,140) .. controls (165,137.24) and (167.24,135) .. (170,135) .. controls (172.76,135) and (175,137.24) .. (175,140) .. controls (175,142.76) and (172.76,145) .. (170,145) .. controls (167.24,145) and (165,142.76) .. (165,140) -- cycle ;
\draw  [fill={rgb, 255:red, 0; green, 0; blue, 0 }  ,fill opacity=1 ] (195,170) .. controls (195,167.24) and (197.24,165) .. (200,165) .. controls (202.76,165) and (205,167.24) .. (205,170) .. controls (205,172.76) and (202.76,175) .. (200,175) .. controls (197.24,175) and (195,172.76) .. (195,170) -- cycle ;
\draw  [fill={rgb, 255:red, 0; green, 0; blue, 0 }  ,fill opacity=1 ]  (195,110) .. controls (195,107.24) and (197.24,105) .. (200,105) .. controls (202.76,105) and (205,107.24) .. (205,110) .. controls (205,112.76) and (202.76,115) .. (200,115) .. controls (197.24,115) and (195,112.76) .. (195,110) -- cycle ;
\draw  [fill={rgb, 255:red, 0; green, 0; blue, 0 }  ,fill opacity=1 ]  (225,140) .. controls (225,137.24) and (227.24,135) .. (230,135) .. controls (232.76,135) and (235,137.24) .. (235,140) .. controls (235,142.76) and (232.76,145) .. (230,145) .. controls (227.24,145) and (225,142.76) .. (225,140) -- cycle ;
\draw  [fill={rgb, 255:red, 0; green, 0; blue, 0 }  ,fill opacity=1 ]  (225,80) .. controls (225,77.24) and (227.24,75) .. (230,75) .. controls (232.76,75) and (235,77.24) .. (235,80) .. controls (235,82.76) and (232.76,85) .. (230,85) .. controls (227.24,85) and (225,82.76) .. (225,80) -- cycle ;
\draw  [fill={rgb, 255:red, 0; green, 0; blue, 0 }  ,fill opacity=1 ]  (255,110) .. controls (255,107.24) and (257.24,105) .. (260,105) .. controls (262.76,105) and (265,107.24) .. (265,110) .. controls (265,112.76) and (262.76,115) .. (260,115) .. controls (257.24,115) and (255,112.76) .. (255,110) -- cycle ;
\draw  [fill={rgb, 255:red, 0; green, 0; blue, 0 }  ,fill opacity=1 ]  (255,50) .. controls (255,47.24) and (257.24,45) .. (260,45) .. controls (262.76,45) and (265,47.24) .. (265,50) .. controls (265,52.76) and (262.76,55) .. (260,55) .. controls (257.24,55) and (255,52.76) .. (255,50) -- cycle ;
\draw  [fill={rgb, 255:red, 0; green, 0; blue, 0 }  ,fill opacity=1 ] (285,20) .. controls (285,17.24) and (287.24,15) .. (290,15) .. controls (292.76,15) and (295,17.24) .. (295,20) .. controls (295,22.76) and (292.76,25) .. (290,25) .. controls (287.24,25) and (285,22.76) .. (285,20) -- cycle ;
\draw  [fill={rgb, 255:red, 0; green, 0; blue, 0 }  ,fill opacity=1 ] (285,80) .. controls (285,77.24) and (287.24,75) .. (290,75) .. controls (292.76,75) and (295,77.24) .. (295,80) .. controls (295,82.76) and (292.76,85) .. (290,85) .. controls (287.24,85) and (285,82.76) .. (285,80) -- cycle ;
\draw  [fill={rgb, 255:red, 0; green, 0; blue, 0 }  ,fill opacity=1 ] (285,140) .. controls (285,137.24) and (287.24,135) .. (290,135) .. controls (292.76,135) and (295,137.24) .. (295,140) .. controls (295,142.76) and (292.76,145) .. (290,145) .. controls (287.24,145) and (285,142.76) .. (285,140) -- cycle ;
\draw  [fill={rgb, 255:red, 0; green, 0; blue, 0 }  ,fill opacity=1 ]  (315,110) .. controls (315,107.24) and (317.24,105) .. (320,105) .. controls (322.76,105) and (325,107.24) .. (325,110) .. controls (325,112.76) and (322.76,115) .. (320,115) .. controls (317.24,115) and (315,112.76) .. (315,110) -- cycle ;
\draw  [fill={rgb, 255:red, 0; green, 0; blue, 0 }  ,fill opacity=1 ] (315,50) .. controls (315,47.24) and (317.24,45) .. (320,45) .. controls (322.76,45) and (325,47.24) .. (325,50) .. controls (325,52.76) and (322.76,55) .. (320,55) .. controls (317.24,55) and (315,52.76) .. (315,50) -- cycle ;
\draw  [fill={rgb, 255:red, 0; green, 0; blue, 0 }  ,fill opacity=1 ] (345,80) .. controls (345,77.24) and (347.24,75) .. (350,75) .. controls (352.76,75) and (355,77.24) .. (355,80) .. controls (355,82.76) and (352.76,85) .. (350,85) .. controls (347.24,85) and (345,82.76) .. (345,80) -- cycle ;
\draw  [fill={rgb, 255:red, 0; green, 0; blue, 0 }  ,fill opacity=1 ] (375,110) .. controls (375,107.24) and (377.24,105) .. (380,105) .. controls (382.76,105) and (385,107.24) .. (385,110) .. controls (385,112.76) and (382.76,115) .. (380,115) .. controls (377.24,115) and (375,112.76) .. (375,110) -- cycle ;
\draw  [fill={rgb, 255:red, 0; green, 0; blue, 0 }  ,fill opacity=1 ] (345,140) .. controls (345,137.24) and (347.24,135) .. (350,135) .. controls (352.76,135) and (355,137.24) .. (355,140) .. controls (355,142.76) and (352.76,145) .. (350,145) .. controls (347.24,145) and (345,142.76) .. (345,140) -- cycle ;
\draw  [fill={rgb, 255:red, 0; green, 0; blue, 0 }  ,fill opacity=1 ] (375,170) .. controls (375,167.24) and (377.24,165) .. (380,165) .. controls (382.76,165) and (385,167.24) .. (385,170) .. controls (385,172.76) and (382.76,175) .. (380,175) .. controls (377.24,175) and (375,172.76) .. (375,170) -- cycle ;
\draw  [fill={rgb, 255:red, 0; green, 0; blue, 0 }  ,fill opacity=1 ] (405,140) .. controls (405,137.24) and (407.24,135) .. (410,135) .. controls (412.76,135) and (415,137.24) .. (415,140) .. controls (415,142.76) and (412.76,145) .. (410,145) .. controls (407.24,145) and (405,142.76) .. (405,140) -- cycle ;
\draw  [fill={rgb, 255:red, 0; green, 0; blue, 0 }  ,fill opacity=1 ] (405,80) .. controls (405,77.24) and (407.24,75) .. (410,75) .. controls (412.76,75) and (415,77.24) .. (415,80) .. controls (415,82.76) and (412.76,85) .. (410,85) .. controls (407.24,85) and (405,82.76) .. (405,80) -- cycle ;
\draw  [fill={rgb, 255:red, 0; green, 0; blue, 0 }  ,fill opacity=1 ] (435,110) .. controls (435,107.24) and (437.24,105) .. (440,105) .. controls (442.76,105) and (445,107.24) .. (445,110) .. controls (445,112.76) and (442.76,115) .. (440,115) .. controls (437.24,115) and (435,112.76) .. (435,110) -- cycle ;
\draw  [fill={rgb, 255:red, 0; green, 0; blue, 0 }  ,fill opacity=1 ] (475,110) .. controls (475,107.24) and (477.24,105) .. (480,105) .. controls (482.76,105) and (485,107.24) .. (485,110) .. controls (485,112.76) and (482.76,115) .. (480,115) .. controls (477.24,115) and (475,112.76) .. (475,110) -- cycle ;
\draw  [fill={rgb, 255:red, 0; green, 0; blue, 0 }  ,fill opacity=1 ] (515,110) .. controls (515,107.24) and (517.24,105) .. (520,105) .. controls (522.76,105) and (525,107.24) .. (525,110) .. controls (525,112.76) and (522.76,115) .. (520,115) .. controls (517.24,115) and (515,112.76) .. (515,110) -- cycle ;
\draw  [fill={rgb, 255:red, 0; green, 0; blue, 0 }  ,fill opacity=1 ] (555,110) .. controls (555,107.24) and (557.24,105) .. (560,105) .. controls (562.76,105) and (565,107.24) .. (565,110) .. controls (565,112.76) and (562.76,115) .. (560,115) .. controls (557.24,115) and (555,112.76) .. (555,110) -- cycle ;

\draw (40,130) node [anchor=north west][inner sep=0.75pt]    {$A^{\bullet-1}(\DD_{5}(5))$};

\end{tikzpicture}
\quad

\tikzset{every picture/.style={line width=0.75pt}} 

\begin{tikzpicture}[x=0.6pt,y=0.6pt,yscale=-1,xscale=1]

\draw [line width=0.75]  [dash pattern={on 4.5pt off 4.5pt}]  (20,110) -- (60,110) ;
\draw [line width=0.75]  [dash pattern={on 4.5pt off 4.5pt}]  (60,110) -- (100,110) ;
\draw  [fill={rgb, 255:red, 0; green, 0; blue, 0 }  ,fill opacity=1 ] (15,110) .. controls (15,107.24) and (17.24,105) .. (20,105) .. controls (22.76,105) and (25,107.24) .. (25,110) .. controls (25,112.76) and (22.76,115) .. (20,115) .. controls (17.24,115) and (15,112.76) .. (15,110) -- cycle ;
\draw [line width=0.75]  [dash pattern={on 4.5pt off 4.5pt}]  (100,110) -- (140,110) ;
\draw [line width=0.75]  [dash pattern={on 4.5pt off 4.5pt}]  (140,110) -- (170,140) ;
\draw [line width=0.75]  [dash pattern={on 4.5pt off 4.5pt}]  (140,110) -- (170,80) ;
\draw [line width=0.75]  [dash pattern={on 4.5pt off 4.5pt}]  (170,80) -- (200,110) ;
\draw [line width=0.75]  [dash pattern={on 4.5pt off 4.5pt}]  (200,110) -- (230,140) ;
\draw [line width=0.75]  [dash pattern={on 4.5pt off 4.5pt}]  (170,140) -- (200,170) ;
\draw [line width=0.75]  [dash pattern={on 4.5pt off 4.5pt}]  (170,140) -- (200,110) ;
\draw [line width=0.75]  [dash pattern={on 4.5pt off 4.5pt}]  (200,170) -- (230,140) ;
\draw [line width=0.75]  [dash pattern={on 4.5pt off 4.5pt}]  (200,110) -- (230,80) ;
\draw [line width=0.75]  [dash pattern={on 4.5pt off 4.5pt}]  (230,80) -- (260,50) ;
\draw [line width=0.75]  [dash pattern={on 4.5pt off 4.5pt}]  (260,50) -- (290,20) ;
\draw [line width=0.75]  [dash pattern={on 4.5pt off 4.5pt}]  (230,140) -- (260,110) ;
\draw [line width=0.75]  [dash pattern={on 4.5pt off 4.5pt}]  (260,110) -- (290,80) ;
\draw [line width=0.75]  [dash pattern={on 4.5pt off 4.5pt}]  (290,80) -- (320,50) ;
\draw [line width=0.75]  [dash pattern={on 4.5pt off 4.5pt}]  (230,80) -- (260,110) ;
\draw [line width=0.75]  [dash pattern={on 4.5pt off 4.5pt}]  (260,50) -- (290,80) ;
\draw [line width=1.50]    (290,20) -- (320,50) ;
\draw [line width=0.75]  [dash pattern={on 4.5pt off 4.5pt}]  (260,110) -- (290,140) ;
\draw [line width=0.75]  [dash pattern={on 4.5pt off 4.5pt}]  (290,80) -- (320,110) ;
\draw [line width=1.50]    (320,50) -- (350,80) ;
\draw [line width=0.75]  [dash pattern={on 4.5pt off 4.5pt}]  (290,140) -- (320,110) ;
\draw [line width=0.75]  [dash pattern={on 4.5pt off 4.5pt}]  (320,110) -- (350,80) ;
\draw [line width=0.75]  [dash pattern={on 4.5pt off 4.5pt}]  (320,110) -- (350,140) ;
\draw [line width=1.50]    (350,80) -- (380,110) ;
\draw [line width=0.75]  [dash pattern={on 4.5pt off 4.5pt}]  (350,140) -- (380,170) ;
\draw [line width=1.50]    (380,110) -- (410,140) ;
\draw [line width=0.75]  [dash pattern={on 4.5pt off 4.5pt}]  (350,140) -- (380,110) ;
\draw  [dash pattern={on 4.5pt off 4.5pt}]  (380,170) -- (410,140) ;
\draw [line width=1.50]    (380,110) -- (410,80) ;
\draw [line width=1.50]    (410,140) -- (440,110) ;
\draw [line width=1.50]    (410,80) -- (440,110) ;
\draw [line width=1.50]    (440,110) -- (480,110) ;
\draw [line width=1.50]    (480,110) -- (520,110) ;
\draw [line width=1.50]    (520,110) -- (560,110) ;
\draw  [fill={rgb, 255:red, 0; green, 0; blue, 0 }  ,fill opacity=1 ] (55,110) .. controls (55,107.24) and (57.24,105) .. (60,105) .. controls (62.76,105) and (65,107.24) .. (65,110) .. controls (65,112.76) and (62.76,115) .. (60,115) .. controls (57.24,115) and (55,112.76) .. (55,110) -- cycle ;
\draw  [fill={rgb, 255:red, 0; green, 0; blue, 0 }  ,fill opacity=1 ]  (95,110) .. controls (95,107.24) and (97.24,105) .. (100,105) .. controls (102.76,105) and (105,107.24) .. (105,110) .. controls (105,112.76) and (102.76,115) .. (100,115) .. controls (97.24,115) and (95,112.76) .. (95,110) -- cycle ;
\draw  [fill={rgb, 255:red, 0; green, 0; blue, 0 }  ,fill opacity=1 ]  (135,110) .. controls (135,107.24) and (137.24,105) .. (140,105) .. controls (142.76,105) and (145,107.24) .. (145,110) .. controls (145,112.76) and (142.76,115) .. (140,115) .. controls (137.24,115) and (135,112.76) .. (135,110) -- cycle ;
\draw  [fill={rgb, 255:red, 0; green, 0; blue, 0 }  ,fill opacity=1 ]  (165,80) .. controls (165,77.24) and (167.24,75) .. (170,75) .. controls (172.76,75) and (175,77.24) .. (175,80) .. controls (175,82.76) and (172.76,85) .. (170,85) .. controls (167.24,85) and (165,82.76) .. (165,80) -- cycle ;
\draw  [fill={rgb, 255:red, 0; green, 0; blue, 0 }  ,fill opacity=1 ]  (165,140) .. controls (165,137.24) and (167.24,135) .. (170,135) .. controls (172.76,135) and (175,137.24) .. (175,140) .. controls (175,142.76) and (172.76,145) .. (170,145) .. controls (167.24,145) and (165,142.76) .. (165,140) -- cycle ;
\draw  [fill={rgb, 255:red, 0; green, 0; blue, 0 }  ,fill opacity=1 ]  (195,170) .. controls (195,167.24) and (197.24,165) .. (200,165) .. controls (202.76,165) and (205,167.24) .. (205,170) .. controls (205,172.76) and (202.76,175) .. (200,175) .. controls (197.24,175) and (195,172.76) .. (195,170) -- cycle ;
\draw  [fill={rgb, 255:red, 0; green, 0; blue, 0 }  ,fill opacity=1 ] (195,110) .. controls (195,107.24) and (197.24,105) .. (200,105) .. controls (202.76,105) and (205,107.24) .. (205,110) .. controls (205,112.76) and (202.76,115) .. (200,115) .. controls (197.24,115) and (195,112.76) .. (195,110) -- cycle ;
\draw  [fill={rgb, 255:red, 0; green, 0; blue, 0 }  ,fill opacity=1 ]  (225,140) .. controls (225,137.24) and (227.24,135) .. (230,135) .. controls (232.76,135) and (235,137.24) .. (235,140) .. controls (235,142.76) and (232.76,145) .. (230,145) .. controls (227.24,145) and (225,142.76) .. (225,140) -- cycle ;
\draw  [fill={rgb, 255:red, 0; green, 0; blue, 0 }  ,fill opacity=1 ]  (225,80) .. controls (225,77.24) and (227.24,75) .. (230,75) .. controls (232.76,75) and (235,77.24) .. (235,80) .. controls (235,82.76) and (232.76,85) .. (230,85) .. controls (227.24,85) and (225,82.76) .. (225,80) -- cycle ;
\draw  [fill={rgb, 255:red, 0; green, 0; blue, 0 }  ,fill opacity=1 ]  (255,110) .. controls (255,107.24) and (257.24,105) .. (260,105) .. controls (262.76,105) and (265,107.24) .. (265,110) .. controls (265,112.76) and (262.76,115) .. (260,115) .. controls (257.24,115) and (255,112.76) .. (255,110) -- cycle ;
\draw  [fill={rgb, 255:red, 0; green, 0; blue, 0 }  ,fill opacity=1 ] (255,50) .. controls (255,47.24) and (257.24,45) .. (260,45) .. controls (262.76,45) and (265,47.24) .. (265,50) .. controls (265,52.76) and (262.76,55) .. (260,55) .. controls (257.24,55) and (255,52.76) .. (255,50) -- cycle ;
\draw  [fill={rgb, 255:red, 0; green, 0; blue, 0 }  ,fill opacity=1 ] (285,20) .. controls (285,17.24) and (287.24,15) .. (290,15) .. controls (292.76,15) and (295,17.24) .. (295,20) .. controls (295,22.76) and (292.76,25) .. (290,25) .. controls (287.24,25) and (285,22.76) .. (285,20) -- cycle ;
\draw  [fill={rgb, 255:red, 0; green, 0; blue, 0 }  ,fill opacity=1 ]  (285,80) .. controls (285,77.24) and (287.24,75) .. (290,75) .. controls (292.76,75) and (295,77.24) .. (295,80) .. controls (295,82.76) and (292.76,85) .. (290,85) .. controls (287.24,85) and (285,82.76) .. (285,80) -- cycle ;
\draw  [fill={rgb, 255:red, 0; green, 0; blue, 0 }  ,fill opacity=1 ]  (285,140) .. controls (285,137.24) and (287.24,135) .. (290,135) .. controls (292.76,135) and (295,137.24) .. (295,140) .. controls (295,142.76) and (292.76,145) .. (290,145) .. controls (287.24,145) and (285,142.76) .. (285,140) -- cycle ;
\draw  [fill={rgb, 255:red, 0; green, 0; blue, 0 }  ,fill opacity=1 ]  (315,110) .. controls (315,107.24) and (317.24,105) .. (320,105) .. controls (322.76,105) and (325,107.24) .. (325,110) .. controls (325,112.76) and (322.76,115) .. (320,115) .. controls (317.24,115) and (315,112.76) .. (315,110) -- cycle ;
\draw  [fill={rgb, 255:red, 0; green, 0; blue, 0 }  ,fill opacity=1 ]  (315,50) .. controls (315,47.24) and (317.24,45) .. (320,45) .. controls (322.76,45) and (325,47.24) .. (325,50) .. controls (325,52.76) and (322.76,55) .. (320,55) .. controls (317.24,55) and (315,52.76) .. (315,50) -- cycle ;
\draw  [fill={rgb, 255:red, 0; green, 0; blue, 0 }  ,fill opacity=1 ]  (345,80) .. controls (345,77.24) and (347.24,75) .. (350,75) .. controls (352.76,75) and (355,77.24) .. (355,80) .. controls (355,82.76) and (352.76,85) .. (350,85) .. controls (347.24,85) and (345,82.76) .. (345,80) -- cycle ;
\draw  [fill={rgb, 255:red, 0; green, 0; blue, 0 }  ,fill opacity=1 ] (375,110) .. controls (375,107.24) and (377.24,105) .. (380,105) .. controls (382.76,105) and (385,107.24) .. (385,110) .. controls (385,112.76) and (382.76,115) .. (380,115) .. controls (377.24,115) and (375,112.76) .. (375,110) -- cycle ;
\draw  [fill={rgb, 255:red, 0; green, 0; blue, 0 }  ,fill opacity=1 ]  (345,140) .. controls (345,137.24) and (347.24,135) .. (350,135) .. controls (352.76,135) and (355,137.24) .. (355,140) .. controls (355,142.76) and (352.76,145) .. (350,145) .. controls (347.24,145) and (345,142.76) .. (345,140) -- cycle ;
\draw  [fill={rgb, 255:red, 0; green, 0; blue, 0 }  ,fill opacity=1 ] (375,170) .. controls (375,167.24) and (377.24,165) .. (380,165) .. controls (382.76,165) and (385,167.24) .. (385,170) .. controls (385,172.76) and (382.76,175) .. (380,175) .. controls (377.24,175) and (375,172.76) .. (375,170) -- cycle ;
\draw  [fill={rgb, 255:red, 0; green, 0; blue, 0 }  ,fill opacity=1 ] (405,140) .. controls (405,137.24) and (407.24,135) .. (410,135) .. controls (412.76,135) and (415,137.24) .. (415,140) .. controls (415,142.76) and (412.76,145) .. (410,145) .. controls (407.24,145) and (405,142.76) .. (405,140) -- cycle ;
\draw  [fill={rgb, 255:red, 0; green, 0; blue, 0 }  ,fill opacity=1 ] (405,80) .. controls (405,77.24) and (407.24,75) .. (410,75) .. controls (412.76,75) and (415,77.24) .. (415,80) .. controls (415,82.76) and (412.76,85) .. (410,85) .. controls (407.24,85) and (405,82.76) .. (405,80) -- cycle ;
\draw  [fill={rgb, 255:red, 0; green, 0; blue, 0 }  ,fill opacity=1 ] (435,110) .. controls (435,107.24) and (437.24,105) .. (440,105) .. controls (442.76,105) and (445,107.24) .. (445,110) .. controls (445,112.76) and (442.76,115) .. (440,115) .. controls (437.24,115) and (435,112.76) .. (435,110) -- cycle ;
\draw  [fill={rgb, 255:red, 0; green, 0; blue, 0 }  ,fill opacity=1 ]  (475,110) .. controls (475,107.24) and (477.24,105) .. (480,105) .. controls (482.76,105) and (485,107.24) .. (485,110) .. controls (485,112.76) and (482.76,115) .. (480,115) .. controls (477.24,115) and (475,112.76) .. (475,110) -- cycle ;
\draw  [fill={rgb, 255:red, 0; green, 0; blue, 0 }  ,fill opacity=1 ]  (515,110) .. controls (515,107.24) and (517.24,105) .. (520,105) .. controls (522.76,105) and (525,107.24) .. (525,110) .. controls (525,112.76) and (522.76,115) .. (520,115) .. controls (517.24,115) and (515,112.76) .. (515,110) -- cycle ;
\draw  [fill={rgb, 255:red, 0; green, 0; blue, 0 }  ,fill opacity=1 ] (555,110) .. controls (555,107.24) and (557.24,105) .. (560,105) .. controls (562.76,105) and (565,107.24) .. (565,110) .. controls (565,112.76) and (562.76,115) .. (560,115) .. controls (557.24,115) and (555,112.76) .. (555,110) -- cycle ;

\draw (440,70) node [anchor=north west][inner sep=0.75pt]    {$A^{\bullet-8 }(\DD_{5}( 1))$};

\end{tikzpicture}      

\caption{Shifted Hasse sub-diagrams corresponding to the fixed-point components $\DD_5(5)$ and $\DD_5(1)$.}
\label{fig:hasse}
	\end{figure}
\end{remark}

Note that these arguments can be redone for an arbitrary generalized Grassmannian of exceptional type $\DE_7, \DE_8, \DF_4, \DG_2$. The recipe is as follows.
\begin{enumerate}
	\item Choose the generalized Grassmannian $\cD(j)$ such that $H_j$ provides the $\C^*$-action of minimal bandwidth. By Lemma \ref{lem:ePsource}, the sink is an isolated point.
	\item Because in the Weyl group of $G$ we have $w_\circ=-\id$, also the source is an isolated fixed point.
	\item Now, use the compass to understand which are the fixed-point components next to the sink and the source. They are the VMRT of $\cD(j)$.
	\item At this point, because the minimal bandwidth of $\cD(j)$ is at most $4$, it may remain to understand the middle fixed-point component. Arguing as in the case of $\DE_6(6)$, one can understand what is the missing fixed-point component.
	\item If we are interested in another RH variety $\cD(k)$, then $\cD(k)$ parametrizes a family of certain subvarieties of $\cD(j)$ and one can induce the fixed-point components of the action on $\cD(k)$ using the ones of $\cD(j)$ as we have done in the section of classic cases.
\end{enumerate}

\begin{remark}
	The same picture of Figure \ref{fig:hasse} appears in \cite{MR3210405}. In the language of Pech, the Hasse diagram describes the $P_1$-orbits in $\DE_6(6)$, which are nothing but vector bundles over the fixed-point components $Y_i$. Our picture describes, in the language of Pech, the $P_6$-orbits in $\DE_6(6)$. In fact, let us note that by Remark \ref{rem:PZgrading} we have $G_0 \subset P_6$ as the reductive part. As pointed out by Perrin in \cite[Proposition 5]{MR1881572} the $P_6$-orbits are of the form $P_6 wP_6/P_6$ for $w \in W(\DE_6)$ and there is an affine fibration
	\[
	f: P_6 wP_6/P_6 \to G_0/(G_0 \cap \conj_w(P_6))\simeq G^\perp /(G^\perp \cap \conj_w(P_6)).
	\]
	Thanks to Theorem \ref{thm:trasvComp}, $Y_i$ are of the form $G^\perp /(G^\perp \cap \conj_w(P_6))$. 
	
	On the other hand we have  the Theorem of Bia\l ynicki-Birula (see the statement in \cite{RomanoW}): define for a fixed-point component $Y \in \cY$ the cells
	\[
	X^\pm(Y):= \left\{ x \in X : \lim_{t \to 0} t^{\pm 1} \cdot x \in Y\right\},
	\] 
	then there is $\C^*$-isomorphism $X^\pm(Y) \simeq N^\pm(Y|X)$ and the natural maps $X^\pm(Y) \to Y$ are $\C^{\nu_\pm(Y)}$-fibrations. Now, thanks to the first formula of \cite{MR718936}, we have 
	\[
	H_{\bullet-2\nu_\pm(Y)}(X^\pm(Y), \Z) \simeq H_{\bullet-2\nu_\pm(Y)}(Y,\Z)
	\]
	and we recover our description in terms of Hasse diagrams via the fixed-point components.
\end{remark}

\bibliographystyle{plain}
\bibliography{referenze.bib}

\begin{thebibliography}{10}

\bibitem{article:bb2}
A.~Bia\l~ynicki Birula.
\newblock On fixed points of torus actions on projective varieties.
\newblock {\em Bull. Acad. Polon. Sci. S\'{e}r. Sci. Math. Astronom. Phys.},
  22:1097--1101, 1974.

\bibitem{article:bb}
Andrzej Bia{\l}ynicki-Birula.
\newblock Some theorems on actions of algebraic groups.
\newblock {\em Ann. of Math. (2)}, 98:480--497, 1973.

\bibitem{bourbaki6}
Nicolas Bourbaki.
\newblock {\em \'{E}l\'ements de math\'ematique. {F}asc. {XXXIV}. {G}roupes et
  alg\`ebres de {L}ie. {C}hapitre {IV}: {G}roupes de {C}oxeter et syst\`emes de
  {T}its. {C}hapitre {V}: {G}roupes engendr\'es par des r\'eflexions.
  {C}hapitre {VI}: syst\`emes de racines}.
\newblock Actualit\'es Scientifiques et Industrielles, No. 1337. Hermann,
  Paris, 1968.

\bibitem{BuWeWi}
Jaros{\l}aw Buczy{\'n}ski, Jaros{\l}aw~A. Wi{\'s}niewski, and Andrzej Weber.
\newblock Algebraic torus actions on contact manifolds.
\newblock {\em To appear in J. Differ. Geom. Preprint ArXiv:1802.05002}, 2018.

\bibitem{MR718936}
J.~B. Carrell and R.~M. Goresky.
\newblock A decomposition theorem for the integral homology of a variety.
\newblock {\em Invent. Math.}, 73(3):367--381, 1983.

\bibitem{Carrell2002}
James~B. Carrell.
\newblock Torus actions and cohomology.
\newblock In {\em Algebraic quotients. {T}orus actions and cohomology. {T}he
  adjoint representation and the adjoint action}, volume 131 of {\em
  Encyclopaedia Math. Sci.}, pages 83--158. Springer, Berlin, 2002.

\bibitem{polyCoxeter}
H.~S.~M. Coxeter.
\newblock {\em Regular polytopes}.
\newblock Dover Publications, Inc., New York, third edition, 1973.

\bibitem{fultonHarris}
William Fulton and Joe Harris.
\newblock {\em Representation theory}, volume 129 of {\em Graduate Texts in
  Mathematics}.
\newblock Springer-Verlag, New York, 1991.

\bibitem{article:gkm}
Mark Goresky, Robert Kottwitz, and Robert MacPherson.
\newblock Equivariant cohomology, {K}oszul duality, and the localization
  theorem.
\newblock {\em Invent. Math.}, 131(1):25--83, 1998.

\bibitem{MR664117}
V.~Guillemin and S.~Sternberg.
\newblock Convexity properties of the moment mapping.
\newblock {\em Invent. Math.}, 67(3):491--513, 1982.

\bibitem{humLie}
James~E. Humphreys.
\newblock {\em Introduction to {L}ie algebras and representation theory},
  volume~9 of {\em Graduate Texts in Mathematics}.
\newblock Springer-Verlag, New York, 1978.

\bibitem{humCoxeter}
James~E. Humphreys.
\newblock {\em Reflection groups and {C}oxeter groups}, volume~29 of {\em
  Cambridge Studies in Advanced Mathematics}.
\newblock Cambridge University Press, Cambridge, 1990.

\bibitem{ManivelCayleyPlane}
Atanas Iliev and Laurent Manivel.
\newblock The {C}how ring of the {C}ayley plane.
\newblock {\em Compos. Math.}, 141(1):146--160, 2005.

\bibitem{iversen}
Birger Iversen.
\newblock A fixed point formula for action of tori on algebraic varieties.
\newblock {\em Invent. Math.}, 16:229--236, 1972.

\bibitem{article:kklv}
Friedrich Knop, Hanspeter Kraft, Domingo Luna, and Thierry Vust.
\newblock Local properties of algebraic group actions.
\newblock In {\em Algebraische {T}ransformationsgruppen und
  {I}nvariantentheorie}, volume~13 of {\em DMV Sem.}, pages 63--75.
  Birkh\"{a}user, Basel, 1989.

\bibitem{ManivelFreud}
J.~M. Landsberg and L.~Manivel.
\newblock The projective geometry of {F}reudenthal's magic square.
\newblock {\em J. Algebra}, 239(2):477--512, 2001.

\bibitem{Landsberg2003}
Joseph~M. Landsberg and Laurent Manivel.
\newblock On the projective geometry of rational homogeneous varieties.
\newblock {\em Comment. Math. Helv.}, 78(1):65--100, 2003.

\bibitem{manivel2020topics}
Laurent Manivel.
\newblock Topics on the {G}eometry of {R}ational {H}omogeneous {S}paces.
\newblock {\em Acta Math. Sin. (Engl. Ser.)}, 36(8):851--872, 2020.

\bibitem{montagard2009regular}
Pierre-Louis Montagard and Nicolas Ressayre.
\newblock Regular lattice polytopes and root systems.
\newblock {\em Bull. Lond. Math. Soc.}, 41(2):227--241, 2009.

\bibitem{campanaSurvey}
Roberto Mu{\~n}oz, Gianluca Occhetta, Luis~E. Sol{\'a}~Conde, Kiwamu Watanabe,
  and Jaros{\l}aw~A. Wi\'sniewski.
\newblock A survey on the {C}ampana-{P}eternell conjecture.
\newblock {\em Rend. Istit. Mat. Univ. Trieste}, 47:127--185, 2015.

\bibitem{occhetta2020high}
Gianluca Occhetta, Eleonora~A. Romano, Luis E.~Sol\'{a} Conde, and Jaros\l
  aw~A. Wi\'{s}niewski.
\newblock High rank torus actions on contact manifolds.
\newblock {\em Selecta Math. (N.S.)}, 27(1):Paper No. 10, 33, 2021.

\bibitem{smallBandwidth}
Gianluca Occhetta, Eleonora~A. Romano, Luis E.~Sol\'a Conde, and Jaros{\l}aw~A.
  Wi\'sniewski.
\newblock Small bandwidth $\mathbb{C}^*$-actions and birational geometry.
\newblock Preprint ArXiv:1911.12129, 2019.

\bibitem{pasquierPrinc}
B.~Pasquier.
\newblock On some smooth projective two-orbit varieties with picard number 1.
\newblock {\em Mathematische Annalen}, 344:963--987, 2009.

\bibitem{MR3210405}
Cl\'{e}lia Pech.
\newblock Quantum product and parabolic orbits in homogeneous spaces.
\newblock {\em Comm. Algebra}, 42(11):4679--4695, 2014.

\bibitem{MR1881572}
Nicolas Perrin.
\newblock Courbes rationnelles sur les vari\'{e}t\'{e}s homog\`enes.
\newblock {\em Ann. Inst. Fourier (Grenoble)}, 52(1):105--132, 2002.

\bibitem{RomanoW}
Eleonora~A. Romano and Jaros{\l}aw~A. Wi{\'s}niewski.
\newblock Adjunction for varieties with a $\mathbb{C}^*$ action.
\newblock {\em arXiv preprint:1904.01896. To appear in Transformation Groups},
  2019.

\bibitem{tevelev2006projective}
Evgueni~A. Tevelev.
\newblock {\em Projective duality and homogeneous spaces}, volume 133 of {\em
  Encyclopaedia of Mathematical Sciences}.
\newblock Springer-Verlag, Berlin, 2005.

\end{thebibliography}

\end{document}